\documentclass[10pt,reqno]{amsart}
\usepackage[foot]{amsaddr}

\usepackage[utf8]{inputenc}
\usepackage[english]{babel}

\usepackage[left=2cm,right=2cm,top=2cm,bottom=2cm]{geometry}
\usepackage{enumitem}
\setlist[itemize]{leftmargin=2em}

\usepackage{graphics}
\usepackage{float}

\numberwithin{equation}{section}

\usepackage{amsmath}	
\usepackage{amssymb}	
\usepackage{mathtools}	

\usepackage[capitalise,noabbrev]{cleveref}
\crefname{equation}{Eq.\!\!}{Eq.\!\!}
\crefname{definition}{D.\!\!}{D.\!\!}
\crefname{lemma}{L.\!\!}{L.\!\!}
\crefname{proposition}{P.\!\!}{P.\!\!}
\crefname{theorem}{T.\!\!}{T.\!\!}
\crefname{corollary}{C.\!\!}{C.\!\!}
\crefname{example}{E.\!\!}{E.\!\!}
\crefname{remark}{R.\!\!}{R.\!\!}

\newtheorem{definition}{Definition}[section]
\newtheorem{lemma}[definition]{Lemma}

\newtheorem{theorem}[definition]{Theorem}
\newtheorem{corollary}[definition]{Corollary}

\newtheorem{remark}[definition]{Remark}

\usepackage{ifthen}
\newcommand{\ifNotEmpty}[2]{\ifthenelse{\equal{#1}{}}{}{#2}}

\renewcommand{\phi}{\varphi}
\renewcommand{\ell}{l}
\renewcommand{\epsilon}{\varepsilon}
\renewcommand{\*}{\hspace{-0.15em}\cdot\hspace{-0.15em}}

\DeclareMathOperator*{\bigcupDot}{\dot{\bigcup}}

\newcommand{\cleq}{\lesssim}
\newcommand{\cgeq}{\gtrsim}
\newcommand{\ceq}{\eqsim}

\renewcommand{\(}{\bigg(}
\renewcommand{\)}{\bigg)}

\newcommand{\set}[1]{\{#1\}}
\newcommand{\Set}[2]{\{#1\,|\,#2\}}
\newcommand{\Pow}[1]{\mathrm{Pow}(#1)}

\newcommand{\N}{\mathbb{N}}
\newcommand{\R}{\mathbb{R}}

\newcommand{\Ball}[3][]{\mathrm{Ball}_{#1}(#2,#3)}

\newcommand{\closureN}[1]{\overline{#1}}
\newcommand{\boundaryN}[1]{\oldpartial #1}
\newcommand{\inflateN}[3][]{{#2}^{#3}_{#1}}
\newcommand{\inflate}[3][]{(#2)^{#3}_{#1}}

\newcommand{\cardN}[1]{\# #1}
\newcommand{\meas}[2][]{|#2|_{#1}}
\newcommand{\diam}[2][]{\mathrm{diam}_{#1}(#2)}
\newcommand{\dist}[3][]{\mathrm{dist}_{#1}(#2,#3)}

\newcommand{\supp}[2][]{\mathrm{supp}_{#1}(#2)}
\newcommand{\suppB}[2][]{\mathrm{supp}_{#1}\bigg(#2\bigg)}
\newcommand{\restrictN}[2]{#1|_{#2}}
\newcommand{\restrict}[2]{(#1)|_{#2}}
\newcommand{\mean}[2][]{\langle#2\rangle_{#1}}

\newcommand{\fDef}[3]{#1: #2 \longrightarrow #3}
\newcommand{\fDefB}[5][]{
\ifNotEmpty{#1}{#1:}
\left\{
\begin{array}{ccc}
#2 & \longrightarrow & #3 \\
#4 & \longmapsto & #5
\end{array}
\right.
}

\newcommand{\floor}[1]{\lfloor #1 \rfloor}

\newcommand{\identity}{\mathrm{id}}
\newcommand{\kronecker}[1]{\delta_{#1}}
\newcommand{\Landau}[1]{\mathcal{O}(#1)}

\let\oldpartial\partial
\newcommand{\partialN}[3][]{\oldpartial_{#2}^{\ifNotEmpty{#1}{(#1)}} #3}
\newcommand{\DN}[2]{\mathrm{D}^{#1} #2}
\newcommand{\gradN}[2][]{\nabla_{#1} #2}
\newcommand{\grad}[2][]{\nabla_{#1}(#2)}
\renewcommand{\div}[2][]{\mathrm{div}_{#1}(#2)}

\newcommand{\I}[4][]{\int\displaylimits_{#2}^{#1} #3 \,\mathrm{d}#4}

\newcommand{\abs}[1]{|#1|}
\newcommand{\seminorm}[2][]{|#2|_{#1}}
\newcommand{\norm}[2][]{\|#2\|_{#1}}
\newcommand{\normB}[2][]{\bigg\|#2\bigg\|_{#1}}
\newcommand{\skalar}[3][]{\langle#2,#3\rangle_{#1}}
\newcommand{\skalarB}[3][]{\bigg\langle#2,#3\bigg\rangle_{#1}}
\newcommand{\bilinear}[3][]{#1(#2,#3)}

\newcommand{\ran}[1]{\mathrm{ran}(#1)}
\newcommand{\rank}[1]{\mathrm{rank}(#1)}

\newcommand{\mvemph}[1]{\boldsymbol{#1}}
\newcommand{\diag}[1]{\mathrm{diag}(#1)}
\newcommand{\detN}[1]{\mathrm{det}\,#1}

\newcommand{\spanN}[1]{\mathrm{span}\,#1}
\newcommand{\dimN}[1]{\mathrm{dim}\,#1}

\newcommand{\Ck}[2]{C^{#1}(#2)}
\newcommand{\CkPw}[2]{C^{#1}_{\mathrm{pw}}(#2)}
\newcommand{\lp}[2]{\ell^{#1}(#2)}
\newcommand{\Lp}[2]{L^{#1}(#2)}
\newcommand{\Hk}[2]{H^{#1}(#2)}
\newcommand{\HkO}[3][]{H^{#2}_{0}(#3 \ifNotEmpty{#1}{,} #1)}
\newcommand{\HkPw}[2]{H^{#1}_{\mathrm{pw}}(#2)}
\newcommand{\Wkp}[3]{W^{#1,#2}(#3)}
\newcommand{\WkpPw}[3]{W^{#1,#2}_{\mathrm{pw}}(#3)}
\newcommand{\Pp}[2]{\mathbb{P}^{#1}(#2)}
\newcommand{\Skp}[3]{\mathbb{S}^{#2,#1}(#3)}
\newcommand{\SkpO}[4][]{\mathbb{S}^{#3,#2}_{0}(#4 \ifNotEmpty{#1}{,} #1)}
\newcommand{\SkpOHarm}[4]{\mathbb{S}_{\mathrm{harm}}(#4)}

\newcommand{\Elements}{\mathcal{T}}
\newcommand{\ElementsA}{\mathcal{A}}
\newcommand{\ElementsB}{\mathcal{B}}
\newcommand{\ElementsC}{\mathcal{C}}
\newcommand{\ElementsD}{\mathcal{D}}
\newcommand{\ElementsS}{\mathcal{S}}
\newcommand{\Nodes}{\mathcal{N}}
\newcommand{\Patch}[2]{#1(#2)}
\newcommand{\h}[1]{h_{#1}}
\newcommand{\hMin}[1]{h_{\min,#1}}
\newcommand{\hMax}[1]{h_{\max,#1}}

\newcommand{\refElement}{\hat{T}}

\newcommand{\Tree}[2][]{\mathbb{T}_{#2}^{\ifNotEmpty{#1}{(#1)}}}
\newcommand{\depth}[1]{\mathrm{depth}(#1)}
\newcommand{\BPart}{\mathbb{P}}
\newcommand{\BPartAdm}{\mathbb{P}_{\mathrm{adm}}}
\newcommand{\BPartSmall}{\mathbb{P}_{\mathrm{small}}}
\newcommand{\HMatrices}[2]{\mathcal{H}(#1,#2)}

\newcommand{\CPcre}{\sigma_{\mathrm{Pcr}}}      
\newcommand{\CShape}{\sigma_{\mathrm{shp}}}     
\newcommand{\CUnif}{\sigma_{\mathrm{unif}}}     
\newcommand{\CCard}{\sigma_{\mathrm{card}}}     
\newcommand{\CAdm}{\sigma_{\mathrm{adm}}}       
\newcommand{\CSmall}{\sigma_{\mathrm{small}}}   
\newcommand{\CExp}{\sigma_{\mathrm{exp}}}       

\newcommand{\SolutionOp}[1]{S_{#1}}
\newcommand{\CutoffFc}[2]{\kappa_{#1}^{#2}}
\newcommand{\CutoffOp}[2]{K_{#1}^{#2}}
\newcommand{\CoarseningOp}[2]{Q_{#1}^{#2}}

\begin{document}

\title[Approximating inverse FEM matrices on non-uniform meshes with $\mathcal{H}$-matrices]{Approximating inverse FEM matrices\\ on non-uniform meshes with $\mathcal{H}$-matrices}

\author[N. Angleitner, M. Faustmann, J.M. Melenk]{Niklas Angleitner, Markus Faustmann, Jens Markus Melenk}

\address{Institute for Analysis and Scientific Computing (Inst. E 101), Vienna University of Technology, Wiedner Hauptstrasse 8-10, 1040 Wien, Austria}

\email{niklas.angleitner@tuwien.ac.at, markus.faustmann@tuwien.ac.at, melenk@tuwien.ac.at}

\date{\today}



\begin{abstract}
We consider the approximation of the inverse of the finite element stiffness matrix in the data sparse $\mathcal{H}$-matrix format. For a large class of shape regular but possibly non-uniform meshes including graded meshes, we prove that the inverse of the stiffness matrix can be approximated in the $\mathcal{H}$-matrix format at an exponential rate in the block rank. Since the storage complexity of the hierarchical matrix is logarithmic-linear and only grows linearly in the block-rank, we obtain an efficient approximation that can be used, e.g., as an approximate direct solver or preconditioner for iterative solvers.
\end{abstract}

\maketitle

\section{Introduction}

Discretizations of elliptic partial differential equations on a domain $\Omega \subseteq \R^d$ using the classic finite element method (FEM) usually produce sparse linear systems of equations $\mvemph{A} \mvemph{x} = \mvemph{b}$ with storage requirements linear in the number of unknowns and linear complexity for the matrix-vector multiplication. However, the direct solution of these systems is computationally more expensive. Therefore, iterative solution methods (e.g., Krylov space methods) are popular in applications, since they only need matrix-vector multiplications, which can be done in linear complexity. A drawback of these methods is that convergence can be slow for matrices with large condition numbers unless a suitable preconditioner is employed. These preconditioners have to be taylored to the problem at hand making black box preconditioners that are based on (approximate) direct solvers particularly interesting. Moreover, if one is interested in solving the same problem with (many) different right-hand sides, a direct solver may be computationally advantageous. 

Hierarchical matrices ($\mathcal{H}$-matrices), introduced in \cite{hackbusch99} and extensively studied in the monograph \cite{Hackbusch_Hierarchical_matrices}, provide a different solution approach to this problem that does not suffer from the drawbacks of classic direct and iterative methods. $\mathcal{H}$-matrices are blockwise low-rank matrices. For suitable block structures and block ranks, storing an $\mathcal{H}$-matrix is of logarithmic-linear complexity. Approximating a given matrix in the $\mathcal{H}$-matrix format thus effects a compression. A main difference to other compression methods such as multipole expansions, \cite{rokhlin85,greengard-rokhlin97}, or wavelet methods, \cite{PSS97,schneider98,tausch-white03}, is that the  $\mathcal{H}$-matrix format allows for an approximate arithmetic. It is possible to add and multiply as well as compute inverses and $LU$-decompositions efficiently in the format, \cite{grasedyck01,GH03,Hackbusch_Hierarchical_matrices}. Therefore, using an $\mathcal{H}$-matrix approximation to the inverse $\mvemph{A}^{-1}$ gives an approximate direct solution method of logarithmic linear complexity that can be applied efficiently to multiple right-hand sides. Moreover, an $LU$-decomposition in the $\mathcal{H}$-matrix format can be used as a black-box preconditioner in iterative solvers, \cite{bebendorf07,GHK08,GKL08}. Nonetheless, we mention that the accuracy in terms of the maximal blockwise rank of the computed approximations to $\mvemph{A}^{-1}$ (or the $LU$-decomposition) using $\mathcal{H}$-matrix arithmetic is not fully understood yet.

In order to explain the numerical success of these approximations, first observed in \cite{grasedyck01}, several works in the literature provide existence results of approximations to the inverse matrices in the $\mathcal{H}$-matrix format. For the inverses of FEM matrices, e.g., see \cite{BH03,bebendorf05,boerm10,Faustmann_H_matrices_FEM} and for inverse BEM matrices, see \cite{FMP16,FMP17}. These analyses are restricted to the case of (quasi)uniform meshes, i.e., all mesh elements have comparable size. In a typical FEM scenario, however, locally refined meshes are employed with mesh elements varying greatly in size in order to account for effects such as locally reduced regularity of the solution. A classic example are graded meshes for the solution of elliptic problems in corner domains, \cite{babuska-kellogg-pitkaeranta79a}.

In this article, we generalize the results of \cite{Faustmann_H_matrices_FEM} for quasiuniform meshes to meshes of so called \emph{locally bounded cardinality} (cf. \cref{Mesh_loc_bd_card}), which includes both uniform meshes and algebraically graded meshes. Our main result states that the inverses of FEM matrices for such meshes can be approximated by hierarchical matrices such that the error converges exponentially in the $\mathcal{H}$-matrix block rank $r$. Given a clustering strategy suitable for non-uniform grids, cf. \cite{Grasedyck_Clustering}, the storage complexity of the $\mathcal{H}$-matrix approximant is of logarithmic linear complexity $\Landau{r N \ln N}$. Moreover, we develop an abstract framework that allows for more general FEM basis functions that do not need to have local supports. In fact, locality is necessary only for a set of \emph{dual functions}, which is a substantially weaker assumption. Finally, we streamline some of the arguments made in \cite{Faustmann_H_matrices_FEM}. While not repeated in this article, we mention that the (mostly algebraic) techniques of \cite[Section~5]{Faustmann_H_matrices_FEM} can be employed in exactly the same way to derive exponentially convergent approximate $LU$-decompositions in the $\mathcal{H}$-matrix format.

The present paper is structured as follows: In \cref{Sec_Main_results} we introduce all necessary definitions and concepts and state our main result, \cref{System_matrix_HMatrix_approx}. \cref{Sec_Proof_of_main_result} is dedicated to the proof of the main result. The main technical contribution is the discrete Caccioppoli-type estimate presented in \cref{Space_SkpOHarm_Cacc}, which is of independent interest. For a certain class of functions, it allows us to bound the $H^1$-seminorm on a given subdomain by the $L^2$-norm on a slightly larger subdomain. Finally, \cref{Sec_Numerical_results} provides a numerical example that illustrates our main result.

Concerning notation: We write ``$a \cleq b$'' iff there exists a constant $C>0$ such that ``$a \leq C b$''. The constant might depend on the space dimension $d$, the domain $\Omega$, the coefficients of the PDE, the shape regularity constant of the mesh, and the polynomial degree of the discrete spline space, but it is \emph{in}dependent of all critical parameters such as the mesh width. We write $a \ceq b$, if there hold both $a \cleq b$ and $a \cgeq b$. Matrices and vectors in linear systems of equations are expressed in boldface letters, e.g., $\mvemph{A} \in \R^{N \times N}$ and $\mvemph{f} \in \R^N$. For all $x \in \R^d$ and $\epsilon > 0$, we write $\Ball[2]{x}{r} := \Set{y \in \R^d}{\norm[2]{y-x} < \epsilon}$ for the Euclidean ball of radius $r$ centered at $x$. The norm of the sequence spaces $\ell^1$ and $\ell^2$ is denoted by $\norm[1]{\cdot}$ and $\norm[2]{\cdot}$. For $k \geq 0$, $q \in [1,\infty]$ and domains $\Omega \subseteq \R^d$, we denote the Sobolev by $\Wkp{k}{q}{\Omega}$. For a given mesh $\Elements$, we denote by $\WkpPw{k}{q}{\Elements}$ the broken Sobolev space consisting of elementwise functions from $W^{k,q}$. For all $v \in \WkpPw{k}{q}{\Elements}$ and $\ElementsB \subseteq \Elements$, we set $\seminorm[\Wkp{k}{q}{\ElementsB}]{v} := (\sum_{T \in \ElementsB} \seminorm[\Wkp{k}{q}{T}]{v}^q)^{1/q}$ and $\seminorm[\Wkp{k}{\infty}{\ElementsB}]{v} := \max_{T \in \ElementsB} \seminorm[\Wkp{k}{\infty}{T}]{v}$. Similarly, $\CkPw{0}{\Elements}$ denotes the space of piecewise continuous functions. Finally, it will facilitate notation on numerous occasions to define the \emph{(discrete) support} of a function $v \in \Lp{2}{\Omega}$ on a mesh $\Elements$ by $\supp[\Elements]{v} := \Set{T \in \Elements}{\,\restrictN{v}{T} \not\equiv 0}$. In particular, we have $\supp[\Elements]{v} \subseteq \Elements$ and $\bigcup\supp[\Elements]{v} \subseteq \R^d$, which slightly differs from the usual definition of a support, namely, $\supp{v} := \overline{\Set{x \in \Omega}{v(x) \neq 0}} \subseteq \R^d$.

\section{Main results} \label{Sec_Main_results}
\subsection{The model problem} \label{SSec_Model_problem}

We investigate the following \emph{model problem}: Let $d \geq 1$ and $\Omega \subseteq \R^d$ be a bounded polyhedral Lipschitz domain. Furthermore, let $a_1 \in \Lp{\infty}{\Omega,\R^{d \times d}}$, $a_2 \in \Lp{\infty}{\Omega,\R^d}$ and $a_3 \in \Lp{\infty}{\Omega,\R}$ be given coefficient functions and $f \in \Lp{2}{\Omega}$ be a given right-hand side. We seek a weak solution $u \in \HkO{1}{\Omega}$ to the following equations:
\begin{equation*}
\begin{array}{rcll}
-\div{a_1 \* \gradN{u}} + a_2 \cdot \gradN{u} + a_3u &=& f & \text{in} \,\, \Omega, \\
u &=& 0 & \text{on} \,\, \boundaryN{\Omega}.
\end{array}
\end{equation*}

In the present work, we restrict ourselves to homogeneous Dirichlet conditions. For the treatment of Neumann and Robin boundary conditions, the same arguments as in \cite{Faustmann_H_matrices_FEM} can be employed.

We assume that $a_1$ is coercive in the sense $\skalar{a_1(x) y}{y} \geq \alpha_1 \norm[2]{y}^2$
for all $x \in \Omega$, $y \in \R^d$ and some constant 
$\alpha_1 > \CPcre^2 (\norm[\Lp{\infty}{\Omega}]{a_2} + \norm[\Lp{\infty}{\Omega}]{a_3}) \geq 0$. 
Here, $\CPcre>0$ denotes the constant in the Poincaré inequality $\norm[\Hk{1}{\Omega}]{\cdot} \leq \CPcre \seminorm[\Hk{1}{\Omega}]{\cdot}$ on $\HkO{1}{\Omega}$.

\begin{definition} \label{Bilinear_form}
We introduce the  bilinear form:
\begin{equation*}
\forall u,v \in \HkO{1}{\Omega}: \quad \quad \bilinear[a]{u}{v} := \skalar[\Lp{2}{\Omega}]{a_1 \gradN{u}}{\gradN{v}} + \skalar[\Lp{2}{\Omega}]{a_2 \* \gradN{u}}{v} + \skalar[\Lp{2}{\Omega}]{a_3 u}{v}.
\end{equation*}
\end{definition}

The weak formulation of the \emph{model problem} reads as follows: Find $u \in \HkO{1}{\Omega}$ such that
\begin{equation*}
\forall v \in \HkO{1}{\Omega}: \quad \quad \bilinear[a]{u}{v} = \skalar[\Lp{2}{\Omega}]{f}{v}.
\end{equation*}

The assumptions on the PDE coefficients imply that the bilinear form $\bilinear[a]{\cdot}{\cdot}$ is continuous and coercive, cf. \cref{Bilinear_form_props}. In particular, the well-known Lax-Milgram Lemma yields the existence of a unique solution $u \in \HkO{1}{\Omega}$.

\subsection{The mesh} \label{SSec_Mesh}

Throughout the text, we consider regular, affine meshes in the following sense:

\begin{definition}[Mesh] \label{Mesh}
A finite set $\Elements \subseteq \Pow{\Omega}$ is a \emph{mesh} if there exists an open simplex $\refElement \subseteq \R^d$ (the \emph{reference element}) such that every \emph{element} $T \in \Elements$ is of the form $T = F_T(\refElement)$, where $\fDef{F_T}{\R^d}{\R^d}$ is an affine diffeomorphism. Furthermore, the elements must be pairwise disjoint, i.e., $\meas{T \cap S} = 0$ for all $T \neq S \in \Elements$, and constitute a partition of $\Omega$, i.e., $\bigcup_{T \in \Elements} \overline{T}  = \overline{\Omega}$. Finally, a mesh must be regular in the sense of \cite{Ciarlet_FEM_Meshes}.
\end{definition}

We call a collection of mesh elements $\ElementsB \subseteq \Elements$ a \emph{cluster}. In the literature on hierarchical matrices, the word \emph{cluster} is typically reserved for collections of vector/matrix indices $I \subseteq \set{1,\dots,N}$. In the present work, however, we deal with collections of mesh elements $\ElementsB \subseteq \Elements$ much more frequently. We also note that both concepts are intimately linked via \cref{Basis_fcts_Patch}.

For every subset $B \subseteq \R^d$, we call the set of neighboring mesh elements 
\begin{equation*}
\Patch{\Elements}{B} := \Set{T \in \Elements}{\closureN{T} \cap \closureN{B} \neq \emptyset} \subseteq \Elements
\end{equation*}
the \emph{patch} of $B$. Similarly, for every cluster $\ElementsB \subseteq \Elements$, we set $\Patch{\Elements}{\ElementsB} := \bigcup_{B \in \ElementsB} \Patch{\Elements}{B} \subseteq \Elements$.

To measure the size of an element $T \in \Elements$, we introduce the local \emph{mesh width} $\h{T} := \sup_{x,y \in T} \norm[2]{y-x}$. The corresponding aggregate mesh widths for a cluster $\ElementsB \subseteq \Elements$ read $\h{\ElementsB} := \hMax{\ElementsB} := \max_{T \in \ElementsB} \h{T}$ and $\hMin{\ElementsB} := \min_{T \in \ElementsB} \h{T}$.

Finally, for every $T \in \Elements$, we denote the center of the largest inscribable ball by $x_T \in T$ (the \emph{incenter}). We assume that $\Elements$ is part of a \emph{shape-regular} family of meshes, i.e., there exists a constant $\CShape \geq 1$ such that
\begin{equation*}
\forall T \in \Elements: \quad \quad \Ball[2]{x_T}{\CShape^{-1} \h{T}} \subseteq T \subseteq \bigcup\Patch{\Elements}{T} \subseteq \Ball[2]{x_T}{\CShape \h{T}}.
\end{equation*}

\begin{definition} \label{Mesh_Metric}
We define the \emph{mesh metric}
\begin{equation*}
\forall T,S \in \Elements: \quad \quad \dist[\Elements]{T}{S} := \norm[2]{x_S-x_T}.
\end{equation*}

For all clusters $\ElementsA,\ElementsB \subseteq \Elements$, we denote the corresponding diameters and distances by
\begin{equation*}
\diam[\Elements]{\ElementsA} := \max_{A_1,A_2 \in \ElementsA} \dist[\Elements]{A_1}{A_2}, \quad \quad \dist[\Elements]{\ElementsA}{\ElementsB} := \min_{\substack{A \in \ElementsA, \\ B \in \ElementsB}} \dist[\Elements]{A}{B}.
\end{equation*}

If $\ElementsA$ or $\ElementsB$ contains only one element, e.g., $\ElementsA = \set{T}$, we drop the enclosing braces and simply write $\dist[\Elements]{T}{\ElementsB} := \dist[\Elements]{\set{T}}{\ElementsB}$. Furthermore, $\diam[\Elements]{T} := \diam[\Elements]{\set{T}} = 0$ by definition of the cluster diameter.
\end{definition}

We refer to \cref{Mesh_Metric_props} for some basic properties of the mesh metric.

Compared to \cite{Faustmann_H_matrices_FEM}, we consider a more general class of meshes. Here, the crucial property is the so called \emph{locally bounded cardinality} defined in the following \cref{Mesh_loc_bd_card}. Note that both \emph{uniform} and \emph{graded} meshes have this property, cf.~\cref{SSec_Unif_graded_mesh}.

\begin{definition} \label{Mesh_loc_bd_card}
A mesh $\Elements \subseteq \Pow{\Omega}$ has \emph{locally bounded cardinality}, if there exists a constant $\CCard \geq 1$ such that
\begin{equation*}
\h{\Elements}^{\CCard} \cleq \hMin{\Elements}, \quad \quad \quad \forall \ElementsB \subseteq \Elements: \quad \cardN{\ElementsB} \cleq \( 1 + \frac{\diam[\Elements]{\ElementsB}}{\h{\ElementsB}} \)^{d\CCard}.
\end{equation*}
\end{definition}

\subsection{The basis- and dual functions}

\begin{definition}[Spline spaces] \label{Space_Skp}
Let $k \geq 0$ and $p \geq 0$. We introduce the finite-dimensional \emph{spline spaces}
\begin{eqnarray*}
\Skp{k}{p}{\Elements} &:=& \Set{v \in \Hk{k}{\Omega}}{\forall T \in \Elements: v \circ F_T \in \Pp{p}{\refElement}}, \\
\SkpO{k}{p}{\Elements} &:=& \Skp{k}{p}{\Elements} \cap \HkO{1}{\Omega},
\end{eqnarray*}
where $\Pp{p}{\refElement} := \spanN{\Set{\refElement \ni x \mapsto x^q}{\norm[1]{q} \leq p}}$ 
denotes the usual space of polynomials of (total) degree $p$ on the reference element.
\end{definition}

The following definition introduces the bases of $\SkpO{1}{p}{\Elements}$ that we consider:

\begin{definition}[Basis with local dual functions] \label{Basis_fcts}
Let $p \geq 1$ and $N := \dimN{\SkpO{1}{p}{\Elements}}$. A basis $\set{\phi_1,\dots,\phi_N} \subseteq \SkpO{1}{p}{\Elements}$ has a system of (local) \emph{dual functions} $\set{\lambda_1,\dots,\lambda_N} \subseteq \Lp{2}{\Omega}$, if, for all $n,m \in \set{1,\dots,N}$ and $\mvemph{x} \in \R^N$, there hold the relations
\begin{equation*}
\skalar[\Lp{2}{\Omega}]{\phi_n}{\lambda_m} = \kronecker{nm}, \quad \quad \quad \normB[\Lp{2}{\Omega}]{\sum_{m=1}^{N} \mvemph{x}_m \lambda_m} \cleq \hMin{\Elements}^{-d/2} \norm[2]{\mvemph{x}}.
\end{equation*}

The implied constant may only depend on $d$, $p$, and the shape regularity of the mesh $\Elements$.
\end{definition}

\begin{remark}
Note that we \emph{do not} assume \emph{local} basis functions $\phi_n$, i.e., $\supp[\Elements]{\phi_n} = \Elements$ is allowed. On the other hand, the dual functions $\lambda_n$ \emph{should} have local supports in order to guarantee competitive memory requirements for the $\mathcal{H}$-matrices (cf.~\cref{H_matrices_Remark}). Furthermore, the specific exponent of $\hMin{\Elements}^{-d/2}$ in the stability bound is not crucial, as it only affects the exponent of the prefactor $N^{\CCard + 2}$ in \cref{System_matrix_HMatrix_approx}.
\end{remark}

The fundamental idea of the present work is to derive properties of matrices from properties of function spaces. 
Naturally, one has to think about the connection between abstract matrix indices $n \in \set{1,\dots,N}$ and corresponding physical subdomains of $\Omega$, which is captured in the following definition.

\begin{definition}[Index patches] \label{Basis_fcts_Patch}
We define the \emph{index patches}
\begin{equation*}
\forall I \subseteq \set{1,\dots,N}: \quad \quad \Patch{\Elements}{I} := \bigcup_{n \in I} \supp[\Elements]{\lambda_n} \subseteq \Elements.
\end{equation*}
\end{definition}

Recall from \cref{SSec_Mesh} that $\Patch{\Elements}{B} \subseteq \Elements$ is the patch of a physical subdomain $B \subseteq \R^d$ and that $\Patch{\Elements}{\ElementsB} \subseteq \Elements$ is the patch of a cluster $\ElementsB \subseteq \Elements$. Now, we also have patches $\Patch{\Elements}{I} \subseteq \Elements$ for collections of matrix indices $I \subseteq \set{1,\dots,N}$. Since all three types of patches follow a common idea, we chose the similarity in notation on purpose.

\subsection{The system matrix} \label{SSec_System_matrix}

Let $\Elements \subseteq \Pow{\Omega}$ be a mesh and $p \geq 1$ a fixed polynomial degree. Let $\SkpO{1}{p}{\Elements} \subseteq \HkO{1}{\Omega}$ be the corresponding spline space. We discretize the model problem from \cref{SSec_Model_problem} by means of the spline space and get the following \emph{discrete model problem}: For given $f \in \Lp{2}{\Omega}$, find $u \in \SkpO{1}{p}{\Elements}$ such that
\begin{equation*}
\forall v \in \SkpO{1}{p}{\Elements}: \quad \quad \bilinear[a]{u}{v} = \skalar[\Lp{2}{\Omega}]{f}{v}.
\end{equation*}

Again, existence and uniqueness of a solution $u \in \SkpO{1}{p}{\Elements}$ follow from \cref{Bilinear_form_props} and the Lax-Milgram Lemma.

As usual, given a basis of the discrete space, the discrete model problem can be rephrased as an equivalent linear system of equations. The bilinear form $\bilinear[a]{\cdot}{\cdot}$ from \cref{Bilinear_form} and the basis functions $\phi_n \in \SkpO{1}{p}{\Elements}$ from \cref{Basis_fcts} compose the governing system matrix.

\begin{definition} \label{System_matrix}
We define the system matrix
\begin{equation*}
\mvemph{A} := (\bilinear[a]{\phi_n}{\phi_m})_{m,n=1}^{N} \in \R^{N \times N}.
\end{equation*}
\end{definition}

Note that the unique solvability of the discrete model problem already ensures that the matrix $\mvemph{A}$ is invertible.

\subsection{Hierarchical matrices}

\begin{definition} \label{Block_partition}
A subset $\BPart \subseteq \Pow{\set{1,\dots,N}} \times \Pow{\set{1,\dots,N}}$ is called a \emph{block partition}, if
\begin{equation*}
\bigcupDot_{(I,J) \in \BPart} I \times J = \set{1,\dots,N} \times \set{1,\dots,N}.
\end{equation*}

Let $\CAdm,\CSmall>0$. A block partition $\BPart$ is called \emph{admissible}, if it can be split into parts
\begin{equation*}
\BPart = \BPartAdm \,\, \dot{\cup} \,\, \BPartSmall
\end{equation*}
with
\begin{equation*}
\begin{array}{lrclcl}
\forall (I,J) \in \BPartAdm: & 0 &<& \diam[\Elements]{\Patch{\Elements}{I}} &\leq& \CAdm \dist[\Elements]{\Patch{\Elements}{I}}{\Patch{\Elements}{J}}, \\
\forall (I,J) \in \BPartSmall: \quad & && \min\set{\cardN{I}, \cardN{J}} &\leq& \CSmall.
\end{array}
\end{equation*}
\end{definition}

Typically, an admissible block partition $\BPart$ is constructed in two stages:

First, the indices $I_{\mathrm{root}} := \set{1,\dots,N}$ are split up into a \emph{(hierarchical) cluster tree} $\Tree{N} := (\Tree[L]{N})_{L \geq 1}$. The first level is $\Tree[1]{N} := \set{I_{\mathrm{root}}}$. Then, given the level $\Tree[L]{N}$, all $I \in \Tree[L]{N}$ with $\cardN{I} > \CSmall$ are split in the form $I = I_1 \dot{\cup} I_2$ with $I_1 \neq \emptyset \neq I_2$ via a predefined \emph{clustering strategy} $I \mapsto (I_1,I_2)$. (See, e.g., \cite{Hackbusch_Hierarchical_matrices} for some examples of such clustering strategies.) The combined set of all such children defines the next layer, $\Tree[L+1]{N}$. Clearly, this process stops after a finite number of layers denoted by $\depth{\Tree{N}}$.

Second, the matrix indices $I_{\mathrm{root}} \times I_{\mathrm{root}}$ are split up into a \emph{(hierarchical) block cluster tree} $\Tree{N \times N} := (\Tree[L]{N \times N})_{L \geq 1}$. Here, the first level is $\Tree[1]{N \times N} := \set{(I_{\mathrm{root}}, I_{\mathrm{root}})}$. Then, given the level $\Tree[L]{N \times N}$, all $(I,J) \in \Tree[L]{N \times N}$ with $\diam[\Elements]{\Patch{\Elements}{I}} > \CAdm \dist[\Elements]{\Patch{\Elements}{I}}{\Patch{\Elements}{J}}$ are split into the children $(I_1,J_1), (I_1,J_2), (I_2,J_1), (I_2,J_2)$, where $I = I_1 \dot{\cup} I_2$ and $J = J_1 \dot{\cup} J_2$ as before. Again, all these children are collected in the layer $\Tree[L+1]{N \times N}$. Finally, the block partition $\BPart$ is just the set of all leaves of $\Tree{N \times N}$.

\begin{definition} \label{H_matrices}
Let $\BPart$ be an admissible block partition and $r \in \N$ a given \emph{block rank bound}. We define the set of \emph{$\mathcal{H}$-matrices} by
\begin{equation*}
\HMatrices{\BPart}{r} := \Set{\mvemph{B} \in \R^{N \times N}}{\forall (I,J) \in \BPartAdm: \exists \mvemph{X} \in \R^{I \times r}, \mvemph{Y} \in \R^{J \times r}: \restrictN{\mvemph{B}}{I \times J} = \mvemph{X} \mvemph{Y}^T}.
\end{equation*}
\end{definition}

\begin{remark} \label{H_matrices_Remark}
By \cite[Lemma~{6.13}]{Hackbusch_Hierarchical_matrices}, the memory requirements to store an $\mathcal{H}$-matrix $\mvemph{B} \in \HMatrices{\BPart}{r}$ can be bounded by the quantity $C_{\mathrm{sparse}}(\Tree{N \times N}) (\CSmall + r) \mathrm{depth}(\Tree{N}) N$, where $C_{\mathrm{sparse}}(\Tree{N \times N})>0$ denotes the so-called \emph{sparsity constant}.

In \cite{Grasedyck_Clustering}, the authors present a \emph{geometrically balanced} clustering strategy that guarantees the upper bounds $C_{\mathrm{sparse}}(\Tree{N \times N}) \cleq 1$ and $\depth{\Tree{N}} \cleq \ln(\hMin{\Elements}^{-1})$. Using the relation $\hMin{\Elements} \cgeq \h{\Elements}^{\CCard}$ from \cref{Mesh_loc_bd_card} for meshes with locally bounded cardinality, we can conclude $\depth{\Tree{N}} \cleq \ln(N)$. In particular, we get an overall bound of $\Landau{r N \ln N}$ for the memory requirements to store the matrix $\mvemph{B}$.

Note that this line of reasoning implicitly assumes that the dual functions $\lambda_n \in \Lp{2}{\Omega}$ from \cref{Basis_fcts} have \emph{local} supports. More precisely, we need $\supp[\Elements]{\lambda_n} \subseteq \Patch{\Elements}{T_n}$ for some $T_n \in \Elements$ and have to ensure that these characteristic elements $T_n$ do not coincide too frequently, i.e. $\cardN{\Set{n}{T_n = T}} \cleq 1$ for all elements $T \in \Elements$.
\end{remark}

\subsection{The main result}

The following theorem is the main result of the present work. It states that inverses of FEM matrices with meshes of locally bounded cardinality can be approximated at an exponential rate by hierarchical matrices.

\begin{theorem} \label{System_matrix_HMatrix_approx}
Let $\Elements \subseteq \Pow{\Omega}$ be a mesh of locally bounded cardinality for some $\CCard \geq 1$ in the sense of \cref{Mesh_loc_bd_card} and $\set{\phi_1,\dots,\phi_N} \subseteq \SkpO{1}{p}{\Elements}$ a basis that has a system of local dual functions (see \cref{Basis_fcts}). Let $\bilinear[a]{\cdot}{\cdot}$ be the elliptic bilinear form from \cref{Bilinear_form} and $\mvemph{A} \in \R^{N \times N}$ be the corresponding Galerkin stiffness matrix (\cref{System_matrix}). Finally, let $\BPart$ be an admissible block partition as in \cref{Block_partition}. Then there exists a constant $\CExp>0$ such that, for every block rank bound $r \in \N$, there exists an $\mathcal{H}$-matrix $\mvemph{B} \in \HMatrices{\BPart}{r}$
with
\begin{equation*}
\norm[2]{\mvemph{A}^{-1} - \mvemph{B}} \cleq N^{\CCard + 2} \exp(-\CExp r^{1/(d\CCard + 1)}).
\end{equation*}
\end{theorem}

Under additional assumptions on the block partition $\BPart$, one can reduce the prefactor from $N^{\CCard + 2}$ to $\ln(N) N^{\CCard}$, see \cref{Block_partition_matrix_norm_Remark}. As shown in \cref{SSec_Unif_graded_mesh}, \emph{uniform} and \emph{algebraically graded} meshes have locally bounded cardinality. In particular, we immediately get the following corollary.

\begin{corollary}
Let $\Elements \subseteq \Pow{\Omega}$ be an algebraically graded mesh with grading exponent $\alpha \geq 1$ (see \cref{Mesh_graded}). Then \cref{System_matrix_HMatrix_approx} holds verbatim with $\CCard = \alpha$.
\end{corollary}

\section{Proof of main result} \label{Sec_Proof_of_main_result}
\subsection{Overview} \label{SSec_Proof_overview}

The techniques employed in the proof of our main result are similar to those developed in \cite{Faustmann_H_matrices_FEM} for uniform meshes. However, some  modifications are necessary to deal with the present case of non-uniform meshes $\Elements$ and (possibly) global basis functions $\phi_n \in \SkpO{1}{p}{\Elements}$. Additionally, we simplify several parts of the previous proof considerably.

1) Before we begin the proof, we give a motivation for the assumptions made in \cref{Mesh_loc_bd_card} and \cref{Basis_fcts}. In \cref{SSec_Unif_graded_mesh}, we present two types of meshes with locally bounded cardinality, namely \emph{uniform} and \emph{graded} meshes. The fact that every uniform mesh has locally bounded cardinality will be used during our proof in \cref{Coarsening_op_single}. The locally bounded cardinality of graded meshes shows that \cref{System_matrix_HMatrix_approx} is applicable for graded meshes in the sense of \cref{Mesh_graded}.

Then, in \cref{SSec_Dual_functions}, we present a practical choice for the dual functions $\lambda_n \in \Lp{2}{\Omega}$ from \cref{Basis_fcts} for a common choice of basis functions $\phi_n \in \SkpO{1}{p}{\Elements}$. The results from this section guarantee that \cref{System_matrix_HMatrix_approx} can be used for many different types of \emph{finite element} bases, including the classic \emph{hat functions}.

2) The starting point for our proof is an explicit representation formula for $\mvemph{A}^{-1}$. Since $\mvemph{A}^{-1}$ represents the act of solving the discretized model problem, it is only natural that the corresponding \emph{discrete solution operator} $\fDef{\SolutionOp{\Elements}}{\Lp{2}{\Omega}}{\SkpO{1}{p}{\Elements}}$ will be involved. Additionally, this endeavor requires the dual functions $\lambda_n \in \Lp{2}{\Omega}$ mentioned earlier. We present the explicit formula for $\mvemph{A}^{-1}$ at the end of \cref{SSec_Rep_formula}.

3) In \cref{SSec_Red_mat_func} we use this formula to go from the ``matrix level'' to the ``function level'': Initially, we reduce the problem of approximating $\mvemph{A}^{-1}$ as a whole to the problem of approximating $\restrictN{\mvemph{A}^{-1}}{I \times J}$ for each admissible block $(I,J) \in \BPartAdm$. (The small blocks $\BPartSmall$ are irrelevant in this matter.) As it turns out, this boils down to the following question:

Given admissible clusters $\ElementsB, \ElementsD \subseteq \Elements$ and a free parameter $L \in \N$, how can we construct a low-dimensional subspace $V_{\ElementsB,\ElementsD,L} \subseteq \Lp{2}{\Omega}$ that contains a good approximant of $\restrict{\SolutionOp{\Elements} f}{\ElementsB}$ for every $f \in \Lp{2}{\Omega}$ with $\supp[\Elements]{f} \subseteq \ElementsD$? More precisely, we want to achieve the bounds (for some fixed $\kappa \geq 1$)
\begin{equation*}
\dimN{V_{\ElementsB,\ElementsD,L}} \cleq L^{\kappa}, \quad \quad \quad \inf_{v \in V_{\ElementsB,\ElementsD,L}} \norm[\Lp{2}{\ElementsB}]{\SolutionOp{\Elements} f - v} \cleq 2^{-L} \norm[\Lp{2}{\ElementsD}]{f}.
\end{equation*}

The remaining sections will give an answer to this very question. Since the construction of $V_{\ElementsB,\ElementsD,L}$ is fairly technical and by no means straightforward, the proof is split into further parts:

4) As the notation ``$V_{\ElementsB,\ElementsD,L}$'' already suggests, the notion of \emph{locality} plays a prominent role in almost all parts of the proof. This is why we introduce so called \emph{inflated clusters}, \emph{discrete cut-off functions}, and the \emph{discrete cut-off operator} in \cref{SSec_Disc_cutoff_op}.

5) In \cref{SSec_Space_SkpOHarm} we investigate an important class of functions for our analysis, the spaces of \emph{locally discrete harmonic functions} $\SkpOHarm{1}{p}{\Elements}{\ElementsB} \subseteq \SkpO{1}{p}{\Elements}$. These subspaces have three important properties: First, for certain $f \in \Lp{2}{\Omega}$, they contain the image $\SolutionOp{\Elements} f$. Second, they are invariant under the influence of their respective discrete cut-off operators. Third, they allow for the \emph{discrete Caccioppoli inequality}, a key ingredient in deriving the asserted error bounds for $V_{\ElementsB,\ElementsD,L}$.

6) Finally, in \cref{SSec_Coarse_ops} we construct the \emph{single-} and \emph{multi-step coarsening operators}. For any given $u \in \SkpOHarm{1}{p}{\Elements}{\inflateN{\ElementsB}{\delta}}$ on the inflated cluster $\inflateN{\ElementsB}{\delta} \supseteq \ElementsB$, the single-step coarsening operator $\CoarseningOp{\ElementsB}{\delta}$ produces a ``coarse'' approximation $\CoarseningOp{\ElementsB}{\delta} u \in \SkpOHarm{1}{p}{\Elements}{\ElementsB}$ with a small approximation error on $\ElementsB$. This is by far the most intricate part of the proof and puts all the aforementioned concepts to use. Afterwards, the multi-step coarsening operator $\CoarseningOp{\ElementsB}{\delta,L}$ is just a combination of $L \in \N$ single-step coarsening operators.

7) In \cref{SSec_Put_together} we merely put all the pieces together and finish the proof of \cref{System_matrix_HMatrix_approx}.

\subsection{Examples of meshes with locally bounded cardinality} \label{SSec_Unif_graded_mesh}

In this subsection, we present two representatives of meshes with locally bounded cardinality (cf.~\cref{Mesh_loc_bd_card}): \emph{Uniform} meshes and \emph{graded} meshes. To verify the locally bounded cardinality property for a given mesh, the following lemma is helpful.

\begin{lemma} \label{Mesh_card}
Let $\Elements \subseteq \Pow{\Omega}$ be a shape-regular mesh as in \cref{Mesh}. Then, there hold the bounds
\begin{equation*}
\frac{1}{\h{\Elements}^d} \cleq \cardN{\Elements}, \quad \quad \quad \forall \ElementsB \subseteq \Elements: \quad \cardN{\ElementsB} \cleq \( 1 + \frac{\diam[\Elements]{\ElementsB}}{\hMin{\ElementsB}} \)^d.
\end{equation*}
\end{lemma}

\begin{proof}
Both estimates follow from the relation $\sum_{T \in \ElementsB} \h{T}^d \ceq \sum_{T \in \ElementsB} \meas{T} = \meas{\bigcup\ElementsB}$ with appropriate $\ElementsB \subseteq \Elements$.
\end{proof}

\begin{definition} \label{Mesh_uniform}
A mesh $\Elements \subseteq \Pow{\Omega}$ is called \emph{uniform}, if there exists a constant $\CUnif \geq 1$ such that
\begin{equation*}
\hMin{\Elements} \leq \h{\Elements} \leq \CUnif \hMin{\Elements}.
\end{equation*}
\end{definition}

Using \cref{Mesh_card} we immediately get the following result:

\begin{lemma} \label{Mesh_uniform_card}
Every uniform mesh $\Elements \subseteq \Pow{\Omega}$ has locally bounded cardinality with $\CCard = 1$.
\end{lemma}

\begin{definition}[Mesh graded towards $\Gamma$] \label{Mesh_graded}
Let $\Elements \subseteq \Pow{\Omega}$ be a mesh and $\Gamma \subseteq \R^d$  satisfy $\Gamma \subseteq \R^d \backslash T$ for all $T \in \Elements$. Furthermore, let $\alpha \geq 1$ be a \emph{grading exponent} and $H>0$ a \emph{coarse mesh width}. We say that $\Elements$ is \emph{graded towards $\Gamma$ with parameters $\alpha,H$}, if there holds 
\begin{equation*}
\forall T \in \Elements: \quad \quad \h{T} \ceq \dist[2]{x_T}{\Gamma}^{1-1/\alpha} H.
\end{equation*}

Here, $x_T$ denotes the incenter of the element $T$ and $\dist[2]{x_T}{\Gamma} = \inf_{\gamma \in \Gamma} \norm[2]{x_T-\gamma}$ is the Euclidean distance between a point and a set.
\end{definition}

The set $\Gamma$ towards which the mesh is graded is usually determined by the given problem. For example, reentrant corners of the domain $\Omega$ or regions of non-smoothness of the data may entail a reduced regularity of the  solution $u$ to the model problem from \cref{SSec_Model_problem}. This usually leads to reduced order of convergence of the finite element approximation on quasiuniform meshes. Choosing the set $\Gamma$ to contain all singularities of the solution as well as choosing the parameter $\alpha$ correctly, one can regain the optimal order of convergence. To a large extent, the shape of $\Gamma$ is irrelevant for our analysis. We only require that the mesh resolve $\Gamma$, i.e., the mesh can only be graded towards points/lines that are part of the mesh skeleton.

\begin{lemma} \label{Mesh_graded_card}
Let $\Elements \subseteq \Pow{\Omega}$ be a  mesh graded towards $\Gamma$ with parameters $\alpha,H$. Then, there hold the bounds $H^{\alpha} \cleq \hMin{\Elements} \leq \h{\Elements} \cleq H$. Furthermore, $\Elements$ has locally bounded cardinality with $\CCard = \alpha$.
\end{lemma}

\begin{proof}
We start with the bounds for $\h{\Elements}$ and $\hMin{\Elements}$: For every $T \in \Elements$, we know from \cref{Mesh} that $\Ball[2]{x_T}{\CShape^{-1} \h{T}} \subseteq T$. Combining this with the assumption $\Gamma \subseteq T^c$ from \cref{Mesh_graded} yields $\dist[2]{x_T}{\Gamma} \geq \h{T}/\CShape$. We conclude $\h{T} \ceq \dist[2]{x_T}{\Gamma}^{1-1/\alpha} H \cgeq \h{T}^{1-1/\alpha} H$ and ultimately $\hMin{\Elements} \cgeq H^{\alpha}$. On the other hand, we have the bound $\h{T} \ceq \dist[2]{x_T}{\Gamma}^{1-1/\alpha} H \leq \sup_{x \in \Omega} \dist[2]{x}{\Gamma}^{1-1/\alpha} H \cleq H$ and thus $\h{\Elements} \cleq H$.

It remains to prove the locally bounded cardinality: Let $\ElementsB \subseteq \Elements$ arbitrary. We fix some element $B \in \ElementsB$ with $b := \dist[2]{x_B}{\Gamma} = \min_{T \in \ElementsB} \dist[2]{x_T}{\Gamma}$ and abbreviate $\Delta b := \diam[\Elements]{\ElementsB}$. Note that there holds the bound $\h{\ElementsB} \ceq (\max_{T \in \ElementsB} \dist[2]{x_T}{\Gamma})^{1-1/\alpha} H \cleq (b + \Delta b)^{1-1/\alpha} H$.

In the case $b \leq \Delta b$ we have the lower bound
\begin{equation*}
\hMin{\ElementsB} \geq \hMin{\Elements} \cgeq H^{\alpha} \cgeq \frac{\h{\ElementsB}^{\alpha}}{(b+\Delta b)^{\alpha-1}} \geq \frac{\h{\ElementsB}^{\alpha}}{(2\Delta b)^{\alpha-1}}.
\end{equation*}

In the remaining case $b > \Delta b$ we get
\begin{equation*}
\hMin{\ElementsB} \ceq H \( \min_{T \in \ElementsB} \dist[2]{x_T}{\Gamma} \)^{1-1/\alpha} = H b^{1-1/\alpha} \cgeq \h{\ElementsB} \( \frac{b}{b+\Delta b} \)^{1-1/\alpha} \geq 2^{1/\alpha-1} \h{\ElementsB}.
\end{equation*}

In particular, both cases lead to the estimate
\begin{equation*}
\cardN{\ElementsB} \stackrel{\cref{Mesh_card}}{\cleq} \( 1 + \frac{\Delta b}{\hMin{\ElementsB}} \)^d \cleq \( 1 + \frac{\Delta b}{\h{\ElementsB}} \)^{d\alpha},
\end{equation*}
which concludes the proof.
\end{proof}

\subsection{Examples of dual functions} \label{SSec_Dual_functions}

In this subsection, we present a way to construct bases of $\SkpO{1}{p}{\Elements}$ that is common in the \emph{finite element} method. This scheme encompasses, in particular, the classic \emph{hat functions} $\phi_n \in \SkpO{1}{1}{\Elements}$ as well as their generalization to $p \geq 1$ (Lagrange elements). Then, we show explicitly how to find a dual system $\set{\lambda_1,\dots,\lambda_N} \subseteq \Lp{2}{\Omega}$ in the sense of \cref{Basis_fcts}.

Let $p \geq 1$, $L := \dimN{\Pp{p}{\refElement}}$ and $N := \dimN{\SkpO{1}{p}{\Elements}}$. Let $\set{\phi_1,\dots,\phi_N} \subseteq \SkpO{1}{p}{\Elements}$ be a basis such that:

\emph{1) Local supports:} For every $n \in \set{1,\dots,N}$, there exists an element $T_n \in \Elements$ such that $T_n \in \supp[\Elements]{\phi_n} \subseteq \Patch{\Elements}{T_n}$.

\emph{2) Simple structure:} There exists a basis of \emph{shape functions} $\set{\hat{\phi}_1,\dots,\hat{\phi}_L} \subseteq \Pp{p}{\refElement}$, which determines the shape of the basis elements. More precisely, for every $n \in \set{1,\dots,N}$ and every $T \in \supp[\Elements]{\phi_n}$, there exists an index $\ell(n,T) \in \set{1,\dots,L}$ such that $\restrictN{\phi_n}{T} = \hat{\phi}_{\ell(n,T)} \circ F_T^{-1}$.

\emph{3) Local distinctness:} The basis functions are \emph{locally distinct} in the following sense: For all $n \neq m \in \set{1,\dots,N}$ and all common $T \in \supp[\Elements]{\phi_n} \cap \supp[\Elements]{\phi_m}$, there holds $\ell(n,T) \neq \ell(m,T)$.

For each basis function $\phi_n$ we fix an element $T_n \in \Elements$ as in 1). Note that a standard scaling argument $T \leftrightarrow \refElement$ readily provides the following relation:
\begin{equation*}
\forall n \in \set{1,\dots,N}: \quad \quad \norm[\Lp{2}{\Omega}]{\phi_n} \ceq \h{T_n}^{d/2}.
\end{equation*}

Now, for the construction of the dual functions $\lambda_n \in \Lp{2}{\Omega}$, let $\set{\hat{\lambda}_1,\dots,\hat{\lambda}_L} \subseteq \Pp{p}{\refElement}$ be the unique set of \emph{dual shape functions}, i.e. $\skalar[\Lp{2}{\refElement}]{\hat{\phi}_{\ell}}{\hat{\lambda}_k} = \kronecker{\ell k}$ for all $\ell,k \in \set{1,\dots,L}$. Then, the \emph{dual function} $\lambda_n \in \Skp{0}{p}{\Elements} \subseteq \Lp{2}{\Omega}$ is defined in a piecewise manner: For every $T \neq T_n$, we set $\restrictN{\lambda_n}{T} := 0$, whereas
\begin{equation*}
\restrictN{\lambda_n}{T_n} := \abs{\detN{\gradN{F_{T_n}}}}^{-1} \* (\hat{\lambda}_{\ell(n,T_n)} \circ F_{T_n}^{-1}).
\end{equation*}

\begin{lemma} \label{Basis_fcts_dual_props}
For all $n,m \in \set{1,\dots,N}$ and $\mvemph{x} \in \R^N$, there holds
\begin{equation*}
\skalar[\Lp{2}{\Omega}]{\phi_n}{\lambda_m} = \kronecker{nm}, \quad \quad \quad \normB[\Lp{2}{\Omega}]{\sum_{m=1}^{N} \mvemph{x}_m \lambda_m} \cleq \hMin{\Elements}^{-d/2} \norm[2]{\mvemph{x}}.
\end{equation*}
\end{lemma}

\begin{proof}
Let $n,m \in \set{1,\dots,N}$. If $T_m \notin \supp[\Elements]{\phi_n}$, we have $m \neq n$ and therefore $\skalar[\Lp{2}{\Omega}]{\phi_n}{\lambda_m} = 0 = \kronecker{nm}$. In the remaining case $T_m \in \supp[\Elements]{\phi_n}$ we get
\begin{equation*}
\skalar[\Lp{2}{\Omega}]{\phi_n}{\lambda_m} = \skalar[\Lp{2}{T_m}]{\phi_n}{\lambda_m} = \skalar[\Lp{2}{\refElement}]{\hat{\phi}_{\ell(n,T_m)}}{\hat{\lambda}_{\ell(m,T_m)}} = \kronecker{\ell(n,T_m) \ell(m,T_m)} = \kronecker{nm}.
\end{equation*}

Next, recall that $\meas{T} \ceq \h{T}^d$ for every element $T$ in a shape-regular mesh $\Elements$. For all $m \in \set{1,\dots,N}$, we compute  
\begin{equation*}
\norm[\Lp{2}{\Omega}]{\lambda_m} = \abs{\detN{\gradN{F_{T_m}}}}^{-1} \norm[\Lp{2}{T_m}]{\hat{\lambda}_{\ell(m,T_m)} \circ F_{T_m}^{-1}} = \meas{\refElement}^{1/2} \meas{T_m}^{-1/2} \norm[\Lp{2}{\refElement}]{\hat{\lambda}_{\ell(m,T_m)}} \ceq \h{T_m}^{-d/2}.
\end{equation*}

Finally, for every $T \in \Elements$, we consider the indices $ms(T) := \Set{m}{T_m = T}$. Due to the duality formula from above, the system $\set{\lambda_1,\dots,\lambda_N} \subseteq \Skp{0}{p}{\Elements}$ is linearly independent. In particular, there must hold $\cardN{ms(T)} \cleq 1$. Now, for every $\mvemph{x} \in \R^N$ and every $T \in \Elements$, we obtain
\begin{equation*}
\normB[\Lp{2}{T}]{\sum_{m=1}^{N} \mvemph{x}_m \lambda_m}^2 = \normB[\Lp{2}{T}]{\sum_{m \in ms(T)} \mvemph{x}_m \lambda_m}^2 \leq \( \sum_{m \in ms(T)} \norm[\Lp{2}{\Omega}]{\lambda_m}^2 \) \( \sum_{m \in ms(T)} \mvemph{x}_m^2 \) \cleq \h{T}^{-d} \sum_{m \in ms(T)} \mvemph{x}_m^2.
\end{equation*}

Summing over all elements $T \in \Elements$ then gives the asserted global stability bound. This concludes the proof.
\end{proof}

\subsection{A representation formula for the inverse system matrix} \label{SSec_Rep_formula}

In this subsection, we develop a representation formula for $\mvemph{A}^{-1}$ in terms of three linear operators: Recall that $\mvemph{A}^{-1}$ represents the action of solving the discrete model problem, so there must be a fundamental connection to the \emph{discrete solution operator} $\fDef{\SolutionOp{\Elements}}{\Lp{2}{\Omega}}{\SkpO{1}{p}{\Elements}}$. Additionally, we need a way to turn coefficient vectors $\mvemph{f} \in \R^N$ into functions $f \in \Lp{2}{\Omega}$ that can be plugged into $S_{\Elements}$. For this purpose, we can use the dual functions $\lambda_n \in \Lp{2}{\Omega}$ from \cref{Basis_fcts} and the corresponding \emph{coordinate mapping} $\fDef{\Lambda}{\R^N}{\Lp{2}{\Omega}}$. Finally, the image $S_{\Elements} \Lambda \mvemph{f} \in \SkpO{1}{p}{\Elements}$ must be converted back to a vector in $\R^N$. A straightforward approach would be to use the inverse $\Phi^{-1}$ of the \emph{coordinate mapping} $\fDef{\Phi}{\R^N}{\SkpO{1}{p}{\Elements}}$ associated with the basis functions $\phi_n \in \SkpO{1}{p}{\Elements}$. But, as it turns out, it is advantageous to use the Hilbert space transpose $\fDef{\Lambda^T}{\Lp{2}{\Omega}}{\R^N}$ instead.

First, let us recall the following classic result:

\begin{lemma} \label{Bilinear_form_props}
The bilinear form $a$ from \cref{Bilinear_form} is coercive and continuous:
\begin{equation*}
\forall u,v \in \HkO{1}{\Omega}: \quad \quad \norm[\Hk{1}{\Omega}]{u}^2 \cleq \bilinear[a]{u}{u}, \quad \quad \bilinear[a]{u}{v} \cleq \norm[\Hk{1}{\Omega}]{u} \norm[\Hk{1}{\Omega}]{v}.
\end{equation*}
\end{lemma}

The precise definitions of $\SolutionOp{\Elements}$, $\Phi$, and $\Lambda$ are given in the following \cref{Sol_op_disc}.

\begin{definition} \label{Sol_op_disc}
Let $\fDef{a}{\HkO{1}{\Omega} \times \HkO{1}{\Omega}}{\R}$ the bilinear form from \cref{Bilinear_form}. For every $f \in \Lp{2}{\Omega}$, denote by $\SolutionOp{\Elements} f \in \SkpO{1}{p}{\Elements}$ the unique function satisfying the variational equality
\begin{equation*}
\forall v \in \SkpO{1}{p}{\Elements}: \quad \quad \bilinear[a]{\SolutionOp{\Elements} f}{v} = \skalar[\Lp{2}{\Omega}]{f}{v}.
\end{equation*}

The linear mapping $\fDef{\SolutionOp{\Elements}}{\Lp{2}{\Omega}}{\SkpO{1}{p}{\Elements}}$ is called \emph{discrete solution operator}.
\end{definition}

Recall from \cref{SSec_System_matrix} that existence and uniqueness of $\SolutionOp{\Elements} f$ are provided by the Lax-Milgram Lemma. Additionally, there holds the a priori bound $\norm[\Hk{1}{\Omega}]{\SolutionOp{\Elements} f} \cleq \norm[\Lp{2}{\Omega}]{f}$.

\begin{definition} \label{Coord_mappings}
Let $\set{\phi_1,\dots,\phi_N} \subseteq \SkpO{1}{p}{\Elements}$ be a basis and $\set{\lambda_1,\dots,\lambda_N} \subseteq \Lp{2}{\Omega}$ a dual system compliant with \cref{Basis_fcts}. We denote the corresponding coordinate mappings by
\begin{equation*}
\fDefB[\Phi]{\R^N}{\SkpO{1}{p}{\Elements}}{\mvemph{x}}{\sum_{n=1}^{N} \mvemph{x}_n \phi_n}, \quad \quad \quad \fDefB[\Lambda]{\R^N}{\Lp{2}{\Omega}}{\mvemph{x}}{\sum_{n=1}^{N} \mvemph{x}_n \lambda_n}.
\end{equation*}
\end{definition}

We summarize the most important properties of $\Phi$ and $\Lambda$ in the following lemma. As usual, we use the notation $\supp{\mvemph{x}} := \Set{n \in \set{1,\dots,N}}{\mvemph{x}_n \neq 0}$ for the \emph{support} of a vector $\mvemph{x} \in \R^N$. Furthermore, recall from \cref{Basis_fcts_Patch} the notation $\Patch{\Elements}{I} \subseteq \Elements$ for all abstract matrix index sets $I \subseteq \set{1,\dots,N}$.

\begin{lemma} \label{Coord_mappings_props}
The Hilbert space transpose of $\Lambda$ is given by the operator
\begin{equation*}
\fDefB[\Lambda^T]{\Lp{2}{\Omega}}{\R^N}{v}{(\skalar[\Lp{2}{\Omega}]{v}{\lambda_n})_{n=1}^{N}}.
\end{equation*}

The restriction of $\Lambda^T$ to the subspace $\SkpO{1}{p}{\Elements} \subseteq \Lp{2}{\Omega}$ coincides with the inverse mapping $\Phi^{-1}$. More precisely, for all $\mvemph{x}, \mvemph{y} \in \R^N$ and all $v \in \SkpO{1}{p}{\Elements}$, there hold the  duality/inversion formulae
\begin{equation*}
\skalar[\Lp{2}{\Omega}]{\Phi\mvemph{x}}{\Lambda\mvemph{y}} = \skalar[2]{\mvemph{x}}{\mvemph{y}}, \quad \quad \quad \Lambda^T \Phi \mvemph{x} = \mvemph{x}, \quad \quad \quad \Phi \Lambda^T v = v.
\end{equation*}

Both $\Lambda$ and $\Lambda^T$ preserve locality: For all $\mvemph{x} \in \R^N$, $v \in \Lp{2}{\Omega}$ and $I \subseteq \set{1,\dots,N}$, we have
\begin{equation*}
\supp[\Elements]{\Lambda \mvemph{x}} \subseteq \Patch{\Elements}{\supp{\mvemph{x}}}, \quad \quad \quad \norm[\lp{2}{I}]{\Lambda^T v} \leq \norm{\Lambda} \norm[\Lp{2}{\Patch{\Elements}{I}}]{v}.
\end{equation*}
\end{lemma}

\begin{proof}
The operator $\Lambda^T$ is indeed the Hilbert space transpose of $\Lambda$: For all $v \in \Lp{2}{\Omega}$ and $\mvemph{x} \in \R^N$, we compute
\begin{equation*}
\skalar[2]{\Lambda^T v}{\mvemph{x}} = \sum_{n=1}^{N} \skalar[\Lp{2}{\Omega}]{v}{\lambda_n} \mvemph{x}_n = \skalarB[\Lp{2}{\Omega}]{v}{\sum_{n=1}^{N} \mvemph{x}_n \lambda_n} = \skalar[\Lp{2}{\Omega}]{v}{\Lambda \mvemph{x}}.
\end{equation*}

The duality formula is a direct consequence of the duality property $\skalar[\Lp{2}{\Omega}]{\phi_n}{\lambda_m} = \kronecker{nm}$ from \cref{Basis_fcts}: For all $\mvemph{x}, \mvemph{y} \in \R^N$, we have
\begin{equation*}
\skalar[\Lp{2}{\Omega}]{\Phi\mvemph{x}}{\Lambda\mvemph{y}} = \sum_{n,m=1}^{N} \mvemph{x}_n \mvemph{y}_m \skalar[\Lp{2}{\Omega}]{\phi_n}{\lambda_m} = \sum_{n=1}^{N} \mvemph{x}_n \mvemph{y}_n = \skalar[2]{\mvemph{x}}{\mvemph{y}}.
\end{equation*}

From this, we immediately get the inversion formula $\Lambda^T \Phi \mvemph{x} = \mvemph{x}$ as well. On the other hand, for every $v \in \SkpO{1}{p}{\Elements}$, there holds $\Phi \Lambda^T v = \Phi \Lambda^T \Phi \Phi^{-1} v = \Phi \Phi^{-1} v = v$.

Next, we turn our attention to the preservation of locality by $\Lambda$:
\begin{equation*}
\forall \mvemph{x} \in \R^N: \quad \quad \supp[\Elements]{\Lambda \mvemph{x}} = \suppB[\Elements]{\sum_{n \in \supp{\mvemph{x}}} \mvemph{x}_n \lambda_n} \subseteq \bigcup_{n \in \supp{\mvemph{x}}} \supp[\Elements]{\lambda_n} \stackrel{\cref{Basis_fcts_Patch}}{=} \Patch{\Elements}{\supp{\mvemph{x}}}.
\end{equation*}

Finally, let $v \in \Lp{2}{\Omega}$ and $I \subseteq \set{1,\dots,N}$. Let $\CutoffFc{}{} \in \Lp{\infty}{\Omega}$ be a (discontinuous) cut-off function with $\restrictN{\CutoffFc{}{}}{\Patch{\Elements}{I}} \equiv 1$ and $\restrictN{\CutoffFc{}{}}{\Elements\backslash\Patch{\Elements}{I}} \equiv 0$. Then,
\begin{equation*}
\norm[\lp{2}{I}]{\Lambda^T v} = \norm[\lp{2}{I}]{\Lambda^T(\CutoffFc{}{} v)} \leq \norm[2]{\Lambda^T(\CutoffFc{}{} v)} \leq \norm{\Lambda^T} \norm[\Lp{2}{\Omega}]{\CutoffFc{}{} v} = \norm{\Lambda} \norm[\Lp{2}{\Patch{\Elements}{I}}]{v},
\end{equation*}
which finishes the proof.
\end{proof}

\begin{lemma} \label{System_matrix_Rep_formula}
The system matrix $\mvemph{A} \in \R^{N \times N}$ from \cref{System_matrix}, the discrete solution operator $\fDef{\SolutionOp{\Elements}}{\Lp{2}{\Omega}}{\SkpO{1}{p}{\Elements}}$ from \cref{Sol_op_disc}, and the coordinate mapping $\fDef{\Lambda}{\R^N}{\Lp{2}{\Omega}}$ from \cref{Coord_mappings} are related via the  representation formula
\begin{equation*}
\forall \mvemph{f} \in \R^N: \quad \quad \mvemph{A}^{-1} \mvemph{f} = \Lambda^T \SolutionOp{\Elements} \Lambda \mvemph{f}.
\end{equation*}
\end{lemma}

\begin{proof}
First, we establish a relationship between $\mvemph{A}$ and $a$ by means of the coordinate mapping $\Phi$:
\begin{equation*}
\forall \mvemph{x}, \mvemph{y} \in \R^N: \quad \quad \skalar[2]{\mvemph{A} \mvemph{x}}{\mvemph{y}} \stackrel{\cref{System_matrix}}{=} \sum_{n,m=1}^{N} \bilinear[a]{\phi_n}{\phi_m} \mvemph{x}_n \mvemph{y}_m \stackrel{\cref{Coord_mappings}}{=} \bilinear[a]{\Phi \mvemph{x}}{\Phi \mvemph{y}}.
\end{equation*}

Now, using the duality and inversion formulae from \cref{Coord_mappings_props}, we get
\begin{equation*}
\forall \mvemph{f},\mvemph{y} \in \R^N: \quad \skalar[2]{\mvemph{A} \Lambda^T \SolutionOp{\Elements} \Lambda \mvemph{f}}{\mvemph{y}} = \bilinear[a]{\Phi \Lambda^T \SolutionOp{\Elements} \Lambda \mvemph{f}}{\Phi\mvemph{y}} = \bilinear[a]{\SolutionOp{\Elements} \Lambda \mvemph{f}}{\Phi\mvemph{y}} \stackrel{\cref{Sol_op_disc}}{=} \skalar[\Lp{2}{\Omega}]{\Lambda \mvemph{f}}{\Phi\mvemph{y}} = \skalar[2]{\mvemph{f}}{\mvemph{y}}.
\end{equation*}

This readily implies the stated representation formula.
\end{proof}

\subsection{Reduction from matrix level to function level} \label{SSec_Red_mat_func}

In this subsection, we rephrase the original \emph{matrix} approximation problem as a \emph{function} approximation problem. This will get rid of abstract matrix indices $I \subseteq \set{1,\dots,N}$ in favor of element clusters $\ElementsB \subseteq \Elements$. The following lemma facilitates a reduction from the full matrix to the individual matrix blocks.

\begin{lemma} \label{Block_partition_matrix_norm}
Let $\BPart \subseteq \Pow{\set{1,\dots,N}} \times \Pow{\set{1,\dots,N}}$ be a block partition. Then there holds the estimate
\begin{equation*}
\forall \mvemph{B} \in \R^{N \times N}: \quad \quad \norm[2]{\mvemph{B}} \leq N^2 \* \max_{(I,J) \in \BPart} \norm[2]{\restrictN{\mvemph{B}}{I \times J}}.
\end{equation*}
\end{lemma}

\begin{proof}
The statement follows from 
\begin{equation*}
\forall \mvemph{x} \in \R^N: \quad \norm[2]{\mvemph{B}\mvemph{x}}^2 = \skalar[2]{\mvemph{B}\mvemph{x}}{\mvemph{B}\mvemph{x}} = \sum_{(I,J) \in \BPart} \skalar[2]{\restrictN{\mvemph{B}}{I \times J} \restrictN{\mvemph{x}}{J}}{\restrict{\mvemph{B}\mvemph{x}}{I}} \leq \( \max_{(I,J) \in \BPart} \norm[2]{\restrictN{\mvemph{B}}{I \times J}} \) \cardN{\BPart} \norm[2]{\mvemph{x}} \norm[2]{\mvemph{B}\mvemph{x}}
\end{equation*}
and the bound $\cardN{\BPart} \leq N^2$, which is valid for \emph{any} partition $\BPart$ of $\set{1,\dots,N} \times \set{1,\dots,N}$.
\end{proof}

\begin{remark} \label{Block_partition_matrix_norm_Remark}
The constant $\Landau{N^2}$ in the upper bound is far from optimal. If one assumes a block partition $\BPart$ stemming from a hierarchical cluster tree $\Tree{N}$, then it can be reduced to $\Landau{\ln N}$: In \cite[Lemma 6.5.8]{Hackbusch_Hierarchical_matrices}, the author showed the bound $\norm[2]{\mvemph{B}} \leq C_{\mathrm{sparse}}(\Tree{N \times N}) \depth{\Tree{N}} \max_{(I,J) \in \BPart} \norm[2]{\restrictN{\mvemph{B}}{I \times J}}$ with the sparsity constant $C_{\mathrm{sparse}}(\Tree{N \times N})$ and the depth of the cluster tree $\depth{\Tree{N}}$. Again, due to \cite{Grasedyck_Clustering}, one can achieve $C_{\mathrm{sparse}}(\Tree{N \times N}) \cleq 1$ and $\depth{\Tree{N}} \cleq \ln(\hMin{\Elements}^{-1}) \cleq \ln(N)$ with a geometrically balanced cluster tree on any mesh satisfying $\hMin{\Elements} \cgeq \h{\Elements}^{\CCard}$.
\end{remark}

The following lemma is the main step in shifting the original problem from matrices to function spaces. Note that the representation formula for $\mvemph{A}^{-1}$ from \cref{System_matrix_Rep_formula} plays a crucial role in its proof.

\begin{lemma} \label{System_matrix_block_approx}
Let $(I,J) \in \BPartAdm$ and $V \subseteq \Lp{2}{\Omega}$ be a finite-dimensional subspace. Then, there exist matrices $\mvemph{X} \in \R^{I \times r}$ and $\mvemph{Y} \in \R^{J \times r}$ of size $r \leq \dimN{V}$, such that there holds the error bound
\begin{equation*}
\norm[2]{\restrictN{\mvemph{A}^{-1}}{I \times J} - \mvemph{X} \mvemph{Y}^T} \leq \norm{\Lambda}^2 \* \sup_{\substack{f \in \Lp{2}{\Omega}: \\ \supp[\Elements]{f} \subseteq \Patch{\Elements}{J}}} \norm[\Lp{2}{\Omega}]{f}^{-1} \* \inf_{v \in V} \norm[\Lp{2}{\Patch{\Elements}{I}}]{\SolutionOp{\Elements} f - v}.
\end{equation*}
\end{lemma}

\begin{proof}
We use the transposed coordinate mapping $\fDef{\Lambda^T}{\Lp{2}{\Omega}}{\R^N}$ from \cref{Coord_mappings_props} to define $\mvemph{V} := \restrict{\Lambda^T V}{I} \subseteq \R^I$. Note that $r := \dimN{\mvemph{V}} \leq \dimN{V}$. Next, let the columns of the matrix $\mvemph{X} \in \R^{I \times r}$ be an $\lp{2}{I}$-orthonormal basis of $\mvemph{V}$. In particular, the product $\mvemph{X}\mvemph{X}^T \in \R^{I \times I}$ represents the $\lp{2}{I}$-orthogonal projection from $\R^I$ onto $\mvemph{V}$. Finally, set $\mvemph{Y} := (\restrictN{\mvemph{A}^{-1}}{I \times J})^T \mvemph{X} \in \R^{J \times r}$.

Now, for every $\mvemph{f} \in \R^N$ with $\supp{\mvemph{f}} \subseteq J$, we get the bound
\begin{equation*}
\begin{array}{rclcl}
\norm[\lp{2}{I}]{(\restrictN{\mvemph{A}^{-1}}{I \times J} - \mvemph{X} \mvemph{Y}^T) \restrictN{\mvemph{f}}{J}} &=& \norm[\lp{2}{I}]{(\mvemph{I} - \mvemph{X} \mvemph{X}^T) \restrict{\mvemph{A}^{-1} \mvemph{f}}{I}} &=& \inf_{\mvemph{v} \in \mvemph{V}} \norm[\lp{2}{I}]{\restrict{\mvemph{A}^{-1} \mvemph{f}}{I} - \mvemph{v}} \\
&\stackrel{\cref{System_matrix_Rep_formula}}{=}& \inf_{v \in V} \norm[\lp{2}{I}]{\Lambda^T(\SolutionOp{\Elements} \Lambda \mvemph{f} - v)} &\stackrel{\cref{Coord_mappings_props}}{\leq}& \norm{\Lambda} \* \inf_{v \in V} \norm[\Lp{2}{\Patch{\Elements}{I}}]{\SolutionOp{\Elements} \Lambda \mvemph{f} - v}.
\end{array}
\end{equation*}

We can divide both sides by $\norm[\lp{2}{J}]{\mvemph{f}}$, take suprema and substitute $f := \Lambda \mvemph{f} \in \Lp{2}{\Omega}$. Finally, we use $\supp[\Elements]{f} = \supp[\Elements]{\Lambda \mvemph{f}} \subseteq \Patch{\Elements}{\supp{\mvemph{f}}} \subseteq \Patch{\Elements}{J}$ and $\norm[\lp{2}{J}]{\mvemph{f}}^{-1} \leq \norm{\Lambda} \norm[\Lp{2}{\Omega}]{f}^{-1}$ to get the desired result.
\end{proof}

A thorough understanding of the preceding lemma is absolutely fundamental for the subsequent sections. Therefore, let us recall its interpretation from \cref{SSec_Proof_overview}:

Let $\ElementsB, \ElementsD \subseteq \Elements$ with $0 < \diam[\Elements]{\ElementsB} \leq \CAdm \dist[\Elements]{\ElementsB}{\ElementsD}$ and $L \in \N$. How can we construct a subspace $V_{\ElementsB,\ElementsD,L} \subseteq \Lp{2}{\Omega}$ of dimension $\dimN{V_{\ElementsB,\ElementsD,L}} \cleq L^{\kappa}$ (for some fixed $\kappa \geq 1$) that satisfies the error bound
\begin{equation*}
\inf_{v \in V_{\ElementsB,\ElementsD,L}} \norm[\Lp{2}{\ElementsB}]{\SolutionOp{\Elements} f - v} \cleq 2^{-L} \norm[\Lp{2}{\ElementsD}]{f},
\end{equation*}
for all source functions $f \in \Lp{2}{\Omega}$ with $\supp[\Elements]{f} \subseteq \ElementsD$?

\subsection{The discrete cut-off operator} \label{SSec_Disc_cutoff_op}

The notion of \emph{cluster inflation} provides a means of enlarging a given cluster by a predefined threshold with respect to the mesh metric $\dist[\Elements]{\cdot}{\cdot}$ from \cref{Mesh_Metric}. This is one of the core concepts in our proof and will be used extensively. We acknowledge this fact with tight notation:

\begin{definition} \label{Mesh_Metric_infl_cluster}
For every cluster $\ElementsB \subseteq \Elements$ and every radius $\delta \geq 0$, we introduce the \emph{inflated cluster}
\begin{equation*}
\inflateN{\ElementsB}{\delta} := \Set{T \in \Elements}{\dist[\Elements]{T}{\ElementsB} \leq \delta}.
\end{equation*}
\end{definition}

We summarize the most important facts about the mesh metric and inflated clusters in the subsequent lemma. We omit the elementary proofs, as they follow directly from the respective definitions.

\begin{lemma} \label{Mesh_Metric_props}
The mesh metric $\dist[\Elements]{\cdot}{\cdot}$ from \cref{Mesh_Metric} defines a metric on $\Elements$. There holds the triangle type inequality
\begin{equation*}
\forall \ElementsA,\ElementsB,\ElementsC \subseteq \Elements: \quad \quad \dist[\Elements]{\ElementsA}{\ElementsC} \leq \dist[\Elements]{\ElementsA}{\ElementsB} + \diam[\Elements]{\ElementsB} + \dist[\Elements]{\ElementsB}{\ElementsC}.
\end{equation*}

For every element $T \in \Elements$ and every neighbor $S \in \Patch{\Elements}{T}$, the distance is bounded by $\dist[\Elements]{T}{S} \leq \CShape \h{T}$. On the other hand, for every $S \in \Elements\backslash\set{T}$, we have the lower bound $\dist[\Elements]{T}{S} \geq \CShape^{-1}(\h{T} + \h{S})$. Additionally, for every cluster $\ElementsB \subseteq \Elements$, there holds $\h{\ElementsB} \leq \max\set{\hMin{\ElementsB}, \CShape \diam[\Elements]{\ElementsB}}$.

When dealing with a second mesh $\ElementsS \subseteq \Pow{\Omega}$, cluster diameters are essentially equivalent:
\begin{equation*}
\forall \ElementsB \subseteq \Elements: \quad \quad \diam[\ElementsS]{\Patch{\ElementsS}{\bigcup\ElementsB}} \leq \diam[\Elements]{\ElementsB} + 2\h{\ElementsB} + 2\h{\Patch{\ElementsS}{\bigcup\ElementsB}}.
\end{equation*}

Finally, consider clusters $\ElementsB \subseteq \ElementsC \subseteq \Elements$ and inflation radii $\delta,\epsilon \geq 0$. Then, $\ElementsB \subseteq \inflateN{\ElementsB}{\delta} \subseteq \inflate{\inflateN{\ElementsB}{\delta}}{\epsilon} \subseteq \inflateN{\ElementsB}{\delta+\epsilon} \subseteq \inflateN{\ElementsC}{\delta+\epsilon}$. For the cluster patch $\Patch{\Elements}{\ElementsB}$ we have the inclusion $\Patch{\Elements}{\ElementsB} \subseteq \inflateN{\ElementsB}{\CShape \h{\ElementsB}}$. We conclude this summary with the bounds $\diam[\Elements]{\inflateN{\ElementsB}{\delta}} \leq \diam[\Elements]{\ElementsB} + 2\delta$ and $\h{\inflateN{\ElementsB}{\delta}} \leq \max\set{\h{\ElementsB}, \CShape \delta}$.
\end{lemma}

For the construction of the cut-off function $\CutoffFc{\ElementsB}{\delta}$ in \cref{Cut_off_fct_disc} we will use a variant of the classic \emph{Clément operator}, \cite{clement75}.

\begin{definition} \label{App_op_Clement}
Let $\Nodes \subseteq \closureN{\Omega}$ be the nodes of the mesh $\Elements$ and denote by $\Set{b_N}{N \in \Nodes} \subseteq \Skp{1}{1}{\Elements}$ the well-known \emph{hat-functions}, i.e. $b_N(M) = \kronecker{NM}$. We write $\mean[T]{v} := \meas{T}^{-1} \I{T}{v}{x} \in \R$ for the mean value of a function $v \in \Lp{2}{\Omega}$ on an element $T \in \Elements$. Now, the \emph{Clément operator} $\fDef{J_{\Elements}}{\Lp{2}{\Omega}}{\Skp{1}{1}{\Elements}}$ is defined in a nodewise fashion: For every $v \in \Lp{2}{\Omega}$, we set $J_{\Elements} v := \sum_{N \in \Nodes} \beta_N b_N$, where the nodal value $\beta_N$ is given by
\begin{equation*}
\beta_N := \frac{1}{\cardN{\Patch{\Elements}{N}}} \sum_{T \in \Patch{\Elements}{N}} \mean[T]{v}.
\end{equation*}
\end{definition}

\begin{lemma} \label{App_op_Clement_props}
The linear operator $J_{\Elements}$ has a local projection property: Given a cluster $\ElementsB \subseteq \Elements$ and a function $v \in \Lp{2}{\Omega}$ with $\restrictN{v}{\Patch{\Elements}{\ElementsB}} \equiv \mathrm{const}$, there holds $\restrict{J_{\Elements} v}{\ElementsB} = \restrictN{v}{\ElementsB}$. Furthermore, $J_{\Elements}$ preserves discrete supports: For every $q \geq 0$ and every $v \in \Skp{0}{q}{\Elements}$, there holds $\supp[\Elements]{J_{\Elements} v} \subseteq \Patch{\Elements}{\supp{v}}$. 
Moreover, $J_{\Elements}$ preserves ranges: For every $v \in \Skp{0}{1}{\Elements}$ with $0 \leq v \leq 1$ there also holds $0 \leq J_{\Elements} v \leq 1$. Finally, we have the stability bound
\begin{equation*}
\forall v \in \Lp{2}{\Omega}: \forall T \in \Elements: \quad \quad \h{T} \seminorm[\Wkp{1}{\infty}{T}]{J_{\Elements} v} \cleq \max_{S \in \Patch{\Elements}{T}} \abs{\mean[T]{v} - \mean[S]{v}}.
\end{equation*}
\end{lemma}

The discretized model problem $\bilinear[a]{u}{v} = \skalar[\Lp{2}{\Omega}]{f}{v}$ was phrased in terms of \emph{global} functions $u,v \in \SkpO{1}{p}{\Elements}$. But if we plug in a function $v$ with local support, e.g., $\supp[\Elements]{v} \subseteq \ElementsB$ for some prescribed cluster $\ElementsB \subseteq \Elements$, we can extract local information about $u$ on $\ElementsB$. This motivates the usage of \emph{discrete cut-off functions}.

\begin{lemma} \label{Cut_off_fct_disc}
Let $\ElementsB \subseteq \Elements$ and $\delta>0$ with $4\CShape^3 \h{\ElementsB} \leq \delta \cleq 1$. Then, there exists a \emph{discrete cut-off function} $\CutoffFc{\ElementsB}{\delta}$ with
\begin{equation*}
\CutoffFc{\ElementsB}{\delta} \in \Skp{1}{1}{\Elements}, \quad \quad \supp[\Elements]{\CutoffFc{\ElementsB}{\delta}} \subseteq \inflateN{\ElementsB}{\delta}, \quad \quad \restrictN{\CutoffFc{\ElementsB}{\delta}}{\ElementsB} \equiv 1, \quad \quad 0 \leq \CutoffFc{\ElementsB}{\delta} \leq 1, \quad \quad \norm[\Wkp{1}{\infty}{\Omega}]{\CutoffFc{\ElementsB}{\delta}} \cleq \frac{1}{\delta}.
\end{equation*}
\end{lemma}

\begin{proof}
We abbreviate $\epsilon := \delta/(4\CShape^2) > 0$ and consider a step function $\CutoffFc{}{} \in \Skp{0}{0}{\Elements}$ defined by
\begin{equation*}
\forall T \in \Elements: \quad \quad \restrictN{\CutoffFc{}{}}{T} := \max\set{0, 1-\dist[\Elements]{T}{\Patch{\Elements}{\ElementsB}}/\epsilon} \in \R.
\end{equation*}

From the definition we immediately get $\supp[\Elements]{\CutoffFc{}{}} \subseteq \inflateN{\Patch{\Elements}{\ElementsB}}{\epsilon}$ and $\restrictN{\CutoffFc{}{}}{\Patch{\Elements}{\ElementsB}} \equiv 1$ as well as $0 \leq \CutoffFc{}{} \leq 1$. (Recall that $\Patch{\Elements}{\ElementsB}$ are all patch elements of $\ElementsB$ and $\inflateN{\Patch{\Elements}{\ElementsB}}{\epsilon}$ is the corresponding inflated cluster by a radius of $\epsilon$.) Next, for every $T \in \Elements$ and every neighbor $S \in \Patch{\Elements}{T}$, we apply the triangle inequality from \cref{Mesh_Metric_props} to the clusters $\set{T}, \set{S}, \Patch{\Elements}{\ElementsB}$ and derive $\dist[\Elements]{T}{\Patch{\Elements}{\ElementsB}} \leq \dist[\Elements]{T}{S} + \dist[\Elements]{S}{\Patch{\Elements}{\ElementsB}}$. (Recall from \cref{Mesh_Metric} that $\diam[\Elements]{S} = 0$, since $\set{S}$ contains only one element.) Exploiting the Lipschitz continuity of $t \mapsto \max\set{0,t}$, we get the error bound
\begin{equation*}
\abs{\restrictN{\CutoffFc{}{}}{T} - \restrictN{\CutoffFc{}{}}{S}} \leq \frac{\abs{\dist[\Elements]{T}{\Patch{\Elements}{\ElementsB}} - \dist[\Elements]{S}{\Patch{\Elements}{\ElementsB}}}}{\epsilon} \stackrel{}{\leq} \frac{\dist[\Elements]{T}{S}}{\epsilon} \stackrel{\cref{Mesh_Metric_props}}{\cleq} \frac{\h{T}}{\epsilon} \ceq \frac{\h{T}}{\delta}.
\end{equation*}

We use the Clément operator $\fDef{J_{\Elements}}{\Lp{2}{\Omega}}{\Skp{1}{1}{\Elements}}$ from \cref{App_op_Clement} to define $\CutoffFc{\ElementsB}{\delta} := J_{\Elements} \CutoffFc{}{} \in \Skp{1}{1}{\Elements}$. For the support of $\CutoffFc{\ElementsB}{\delta}$ we compute
\begin{equation*}
\supp[\Elements]{\CutoffFc{\ElementsB}{\delta}} \stackrel{\cref{App_op_Clement_props}}{\subseteq} \Patch{\Elements}{\supp[\Elements]{\CutoffFc{}{}}} \subseteq \Patch{\Elements}{\inflateN{\Patch{\Elements}{\ElementsB}}{\epsilon}} \stackrel{\cref{Mesh_Metric_props}}{\subseteq} \inflateN{\ElementsB}{(1+\CShape^2)(\CShape \h{\ElementsB} + \epsilon)} \subseteq \inflateN{\ElementsB}{2\CShape^3 \h{\ElementsB} + \delta/2} \stackrel{\delta \cgeq \h{\ElementsB}}{\subseteq} \inflateN{\ElementsB}{\delta}.
\end{equation*}

From \cref{App_op_Clement_props} and $\restrictN{\CutoffFc{}{}}{\Patch{\Elements}{\ElementsB}} \equiv 1$ we get $\restrictN{\CutoffFc{\ElementsB}{\delta}}{\ElementsB} \equiv 1$. Moreover, $0 \leq \CutoffFc{}{} \leq 1$ yields $0 \leq \CutoffFc{\ElementsB}{\delta} \leq 1$. This implies, in particular, $\norm[\Lp{\infty}{\Omega}]{\CutoffFc{\ElementsB}{\delta}} \leq 1 \cleq \delta^{-1}$, where we used the assumption $\delta \cleq 1$. The remaining bound $\seminorm[\Wkp{1}{\infty}{\Omega}]{\CutoffFc{\ElementsB}{\delta}} \cleq \delta^{-1}$ follows from
\begin{equation*}
\forall T \in \Elements: \quad \quad \h{T} \seminorm[\Wkp{1}{\infty}{T}]{\CutoffFc{\ElementsB}{\delta}} \stackrel{\cref{App_op_Clement_props}}{\cleq} \max_{S \in \Patch{\Elements}{T}} \abs{\restrictN{\CutoffFc{}{}}{T} - \restrictN{\CutoffFc{}{}}{S}} \cleq \frac{\h{T}}{\delta}.
\end{equation*}

This finishes the proof.
\end{proof}

Given a cluster $\ElementsB \subseteq \Elements$ and a distance $\delta>0$, the discrete cut-off function $\CutoffFc{\ElementsB}{\delta}$ allows us to ``restrict'' a function $v \in \Skp{1}{p}{\Elements}$ to the subdomain $\bigcup\inflateN{\ElementsB}{\delta} \subseteq \Omega$ while preserving continuity. This can be achieved by simply multiplying $v$ with $\CutoffFc{\ElementsB}{\delta}$. Note that the product $\CutoffFc{\ElementsB}{\delta} v$ has polynomial degree $p+1$, rather than $p$. To mitigate this drawback, we can simply re-interpolate the result with an operator of order $p$.

\begin{definition} \label{App_op_Lagrange}
Let $p \geq 1$ and denote by $\fDef{\hat{I}^p}{\Ck{0}{\closureN{\refElement}}}{\Pp{p}{\refElement}}$ the \emph{(local) Lagrange interpolation operator} on the reference element $\refElement$. The \emph{(global) Lagrange interpolation operator} $\fDef{I_{\Elements}^p}{\CkPw{0}{\Elements}}{\Skp{0}{p}{\Elements}}$ is defined in a piecewise manner: For every $v \in \CkPw{0}{\Elements}$ and every $T \in \Elements$, we set
\begin{equation*}
\restrict{I_{\Elements}^p v}{T} := \hat{I}^p(v \circ F_T) \circ F_T^{-1}.
\end{equation*}
\end{definition}

In order to derive a useful stability estimate for $I_{\Elements}^p$, we use a standard inverse inequality (see, e.g., \cite{DFGHS01}).

\begin{lemma} \label{Inverse_inequality}
Let $k \geq \ell \geq 0$, $q \in [1,\infty]$ and $p \geq 0$. Then, for all discrete functions $v \in \Skp{0}{p}{\Elements}$ and all elements $T \in \Elements$, there holds the \emph{inverse inequality}
\begin{equation*}
\h{T}^k \seminorm[\Wkp{k}{q}{T}]{v} \cleq \h{T}^{\ell} \seminorm[\Wkp{\ell}{q}{T}]{v}.
\end{equation*}
\end{lemma}

The properties of the Lagrange interpolation operator $I_{\Elements}^p$ are very similar to those of the Clément operator $J_{\Elements}$ from \cref{App_op_Clement}. For the sake of completeness, we include them in the following lemma.

\begin{lemma} \label{App_op_Lagrange_props}
Let $p \geq 1$. The linear operator $I_{\Elements}^p$ has a local projection property: Given a cluster $\ElementsB \subseteq \Elements$ and a function $v \in \CkPw{0}{\Elements}$ with $v \in \Skp{0}{p}{\ElementsB}$, there holds $\restrict{I_{\Elements}^p v}{\ElementsB} = \restrictN{v}{\ElementsB}$. Furthermore, $I_{\Elements}^p$ preserves global continuity and homogeneous boundary values: For every $v \in \Ck{0}{\closureN{\Omega}}$, there holds $I_{\Elements}^p v \in \Skp{1}{p}{\Elements}$. Similarly, if $v \in \Ck{0}{\closureN{\Omega}}$ with $\restrictN{v}{\boundaryN{\Omega}} \equiv 0$, then $I_{\Elements}^p v \in \SkpO{1}{p}{\Elements}$. Moreover, $I_{\Elements}^p$ preserves discrete supports: For every $q \geq 0$ and every $v \in \Skp{0}{q}{\Elements}$, we have $\supp[\Elements]{I_{\Elements}^p v} \subseteq \supp[\Elements]{v}$. Finally, for all $q \geq 0$, $v \in \Skp{0}{q}{\Elements}$ and $T \in \Elements$, there hold the following stability and error estimates (with constants depending on $q$):
\begin{equation*}
\begin{array}{lrcl}
\forall m \in \set{0,\dots,p+1}: \quad & \seminorm[\Hk{m}{T}]{I_{\Elements}^p v} &\cleq& \seminorm[\Hk{m}{T}]{v}, \\
 & \sum_{\ell=0}^{p+1} \h{T}^{\ell} \seminorm[\Hk{\ell}{T}]{(\identity - I_{\Elements}^p)(v)} &\cleq& \h{T}^{p+1} \seminorm[\Hk{p+1}{T}]{v}.
\end{array}
\end{equation*}
\end{lemma}

\begin{proof}
We briefly sketch the proof of the stability and error bounds: The mapping $v \mapsto \norm[\Lp{2}{\refElement}]{\hat{I}^p v} + \seminorm[\Hk{p+1}{\refElement}]{v}$ defines a norm on the finite-dimensional space $\Pp{q}{\refElement}$. Hence, by norm equivalence, $\norm[\Hk{p+1}{\refElement}]{v} \cleq \norm[\Lp{2}{\refElement}]{\hat{I}^p v} + \seminorm[\Hk{p+1}{\refElement}]{v}$ for all $v \in \Pp{q}{\refElement}$. Inserting $v := w-\hat{I}^p w$ for arbitrary $w \in \Pp{q}{\refElement}$ results in the bound $\norm[\Hk{p+1}{\refElement}]{w-\hat{I}^p w} \cleq \seminorm[\Hk{p+1}{\refElement}]{w}$. Finally, a standard scaling argument $\refElement \leftrightarrow T$ yields the desired error estimate on $T$. As for the stability bound, we perform a straightforward triangle inequality on $T$, reuse the already proven error bound and finish off with the inverse inequality from \cref{Inverse_inequality}.
\end{proof}

\begin{remark}
The fact that $I_{\Elements}^p$ preserves global continuity and homogeneous boundary values hinges on an implicit assumption about the (local) interpolation points used by the local Lagrange interpolation operator $\hat{I}^p$. Recall from \cref{Mesh} that the reference element $\refElement \subseteq \R^d$ is a simplex and thus delimited by $d+1$ hyperplanes. The interpolation points on each hyperplane $\hat{E}$ must be unisolvent for the space $\Pp{p}{\hat{E}}$. Then, in particular, every polynomial $v \in \Pp{p}{\refElement}$ vanishing at the interpolation points in $\hat{E}$ must already vanish everywhere on $\hat{E}$. This property readily implies that homogeneous boundary values are preserved by the global operator $I_{\Elements}^p$. Finally, the distribution of interpolation points on each hyperplane $\hat{E}$ must be ``symmetric''. More precisely, if two elements $T_1,T_2 \in \Elements$ share a common hyperplane, we require the corresponding interpolation points to align perfectly. In this case, using the same argument as before, the operator $I_{\Elements}^p$ preserves global continuity indeed.
\end{remark}

As our next step, we encapsulate the aforementioned ``cut-off'' process in a linear operator.

\begin{definition} \label{Cut_off_op_disc}
Let $\ElementsB \subseteq \Elements$ and $\delta>0$ with $4\CShape^3 \h{\ElementsB} \leq \delta \cleq 1$ and denote by $\CutoffFc{\ElementsB}{\delta} \in \Skp{1}{1}{\Elements}$ the discrete cut-off function from \cref{Cut_off_fct_disc}. Furthermore, denote by $\fDef{I_{\Elements}^p}{\CkPw{0}{\Elements}}{\Skp{0}{p}{\Elements}}$ the Lagrange interpolation operator from \cref{App_op_Lagrange}. We define the \emph{discrete cut-off operator}
\begin{equation*}
\fDefB[\CutoffOp{\ElementsB}{\delta}]{\Skp{1}{p}{\Elements}}{\Skp{1}{p}{\Elements}}{v}{I_{\Elements}^p(\CutoffFc{\ElementsB}{\delta} v)}.
\end{equation*}
\end{definition}

The discrete cut-off operator $\CutoffOp{\ElementsB}{\delta}$ inherits its core properties from $I_{\Elements}^p$.

\begin{lemma} \label{Cut_off_op_disc_props}
Let $\ElementsB \subseteq \Elements$ and $\delta>0$ with $4\CShape^3 \h{\ElementsB} \leq \delta \cleq 1$. For all $v \in \Skp{1}{p}{\Elements}$, the linear operator $\CutoffOp{\ElementsB}{\delta}$ has the cut-off property $\supp[\Elements]{\CutoffOp{\ElementsB}{\delta} v} \subseteq \inflateN{\ElementsB}{\delta}$ and the local projection property $\restrict{\CutoffOp{\ElementsB}{\delta} v}{\ElementsB} = \restrictN{v}{\ElementsB}$. Furthermore, $\CutoffOp{\ElementsB}{\delta}$ preserves homogeneous boundary values: For all $v \in \SkpO{1}{p}{\Elements}$, there holds $\CutoffOp{\ElementsB}{\delta} v \in \SkpO{1}{p}{\Elements}$. Finally, for every $v \in \Skp{1}{p}{\Elements}$ and every $T \in \Elements$, there holds the local stability estimate
\begin{equation*}
\norm[\Lp{2}{T}]{\CutoffOp{\ElementsB}{\delta} v} + \delta \seminorm[\Hk{1}{T}]{\CutoffOp{\ElementsB}{\delta} v} \cleq \norm[\Lp{2}{T}]{v} + \delta \seminorm[\Hk{1}{T}]{v}.
\end{equation*}
\end{lemma}

\begin{proof}
The cut-off property, the local projection property and the preservation of homogeneous boundary values follow directly from \cref{App_op_Lagrange_props} and \cref{Cut_off_fct_disc}. Finally, let $v \in \Skp{1}{p}{\Elements}$ and $T \in \Elements$. Note that $\CutoffFc{\ElementsB}{\delta} v \in \Skp{1}{p+1}{\Elements}$, i.e., we can use the stability estimate from \cref{App_op_Lagrange_props}:
\begin{equation*}
\sum_{\ell=0}^{1} \delta^{\ell} \seminorm[\Hk{\ell}{T}]{\CutoffOp{\ElementsB}{\delta} v} \cleq \sum_{\ell=0}^{1} \delta^{\ell} \seminorm[\Hk{\ell}{T}]{\CutoffFc{\ElementsB}{\delta} v} \cleq \sum_{\ell=0}^{1} \delta^{\ell} \sum_{i=0}^{\ell} \seminorm[\Wkp{\ell-i}{\infty}{T}]{\CutoffFc{\ElementsB}{\delta}} \seminorm[\Hk{i}{T}]{v} \stackrel{\cref{Cut_off_fct_disc}}{\cleq} \sum_{\ell=0}^{1} \delta^{\ell} \seminorm[\Hk{\ell}{T}]{v}.
\end{equation*}
\end{proof}

\subsection{The spaces of locally discrete harmonic functions} \label{SSec_Space_SkpOHarm}

In this subsection, we introduce the spaces of \emph{locally discrete harmonic functions}. As we already mentioned in \cref{SSec_Proof_overview}, they are chosen for three main reasons: To begin with, they fit in seamlessly with the discrete solution operator $\fDef{\SolutionOp{\Elements}}{\Lp{2}{\Omega}}{\SkpO{1}{p}{\Elements}}$ from \cref{Sol_op_disc}. Furthermore, as specified in \cref{Space_SkpOHarm_props}, they are invariant with respect to the discrete cut-off operators $\fDef{\CutoffOp{\ElementsB}{\delta}}{\Skp{1}{p}{\Elements}}{\Skp{1}{p}{\Elements}}$ from \cref{Cut_off_op_disc}. But most importantly, they contain functions whose $H^1$-norms can be bounded by $L^2$-norms with constants independent of $\h{}$, i.e., a \emph{discrete Caccioppoli inequality}.

\begin{definition} \label{Space_SkpOHarm}
For every $\ElementsB \subseteq \Elements$, we define the space of \emph{locally discrete harmonic functions}
\begin{equation*}
\SkpOHarm{1}{p}{\Elements}{\ElementsB} := \Set{u \in \SkpO{1}{p}{\Elements}}{\forall v \in \SkpO{1}{p}{\Elements} \,\, \text{with} \,\, \supp[\Elements]{v} \subseteq \ElementsB: \bilinear[a]{u}{v} = 0} \subseteq \SkpO{1}{p}{\Elements}.
\end{equation*}
\end{definition}

We summarize the first two main features of the spaces $\SkpOHarm{1}{p}{\Elements}{\ElementsB}$ in the next lemma, namely their relationships to the discrete solution operator $\fDef{\SolutionOp{\Elements}}{\Lp{2}{\Omega}}{\SkpO{1}{p}{\Elements}}$ and the discrete cut-off operators $\fDef{\CutoffOp{\ElementsB}{\delta}}{\Skp{1}{p}{\Elements}}{\Skp{1}{p}{\Elements}}$.

\begin{lemma} \label{Space_SkpOHarm_props}
The spaces of locally discrete harmonic functions are nested in the sense
\begin{equation*}
\forall \ElementsB \subseteq \ElementsB^+ \subseteq \Elements: \quad \quad \SkpOHarm{1}{p}{\Elements}{\ElementsB^+} \subseteq \SkpOHarm{1}{p}{\Elements}{\ElementsB}.
\end{equation*}

Furthermore, for all clusters $\ElementsB,\ElementsD \subseteq \Elements$ with $\ElementsB \cap \ElementsD = \emptyset$, the operator $\SolutionOp{\Elements}$ has the  mapping property
\begin{equation*}
\forall f \in \Lp{2}{\Omega} \,\, \text{with} \,\, \supp[\Elements]{f} \subseteq \ElementsD: \quad \quad \SolutionOp{\Elements} f \in \SkpOHarm{1}{p}{\Elements}{\ElementsB}.
\end{equation*}

Finally, for all $\ElementsB \subseteq \Elements$ and all $\delta>0$ with $4\CShape^3 \h{\ElementsB} \leq \delta \cleq 1$, we have the invariance
\begin{equation*}
\forall u \in \SkpOHarm{1}{p}{\Elements}{\ElementsB}: \quad \quad \CutoffOp{\ElementsB}{\delta} u \in \SkpOHarm{1}{p}{\Elements}{\ElementsB}.
\end{equation*}
\end{lemma}

\begin{proof}
The inclusion $\SkpOHarm{1}{p}{\Elements}{\ElementsB^+} \subseteq \SkpOHarm{1}{p}{\Elements}{\ElementsB}$ follows directly from the definition of the spaces. As for the mapping properties of $\SolutionOp{\Elements}$, let $f \in \Lp{2}{\Omega}$ with $\supp[\Elements]{f} \subseteq \ElementsD$. Then, for every $v \in \SkpO{1}{p}{\Elements}$ with $\supp[\Elements]{v} \subseteq \ElementsB$, we have
\begin{equation*}
\bilinear[a]{\SolutionOp{\Elements} f}{v} \stackrel{\cref{Sol_op_disc}}{=} \skalar[\Lp{2}{\ElementsD \cap \ElementsB}]{f}{v} \stackrel{\ElementsB \cap \ElementsD = \emptyset}{=} 0.
\end{equation*}

Finally, consider a function $u \in \SkpOHarm{1}{p}{\Elements}{\ElementsB}$ and an arbitrary $v \in \SkpO{1}{p}{\Elements}$ with $\supp[\Elements]{v} \subseteq \ElementsB$. Then, 
\begin{eqnarray*}
\bilinear[a]{\CutoffOp{\ElementsB}{\delta} u}{v} &\stackrel{\cref{Bilinear_form}}{=}& \skalar[\Lp{2}{\ElementsB}]{a_1 \nabla \CutoffOp{\ElementsB}{\delta} u}{\gradN{v}} + \skalar[\Lp{2}{\ElementsB}]{a_2 \cdot \nabla\CutoffOp{\ElementsB}{\delta} u}{v} + \skalar[\Lp{2}{\ElementsB}]{a_3 \CutoffOp{\ElementsB}{\delta} u}{v} \\
&\stackrel{\cref{Cut_off_op_disc_props}}{=}& \skalar[\Lp{2}{\ElementsB}]{a_1 \gradN{u}}{\gradN{v}} + \skalar[\Lp{2}{\ElementsB}]{a_2 \cdot \gradN{u}}{v} + \skalar[\Lp{2}{\ElementsB}]{a_3 u}{v} \\
&=& \bilinear[a]{u}{v} \\
&=& 0.
\end{eqnarray*}

This gives $\CutoffOp{\ElementsB}{\delta} u \in \SkpOHarm{1}{p}{\Elements}{\ElementsB}$, which concludes the proof.
\end{proof}

Next, we turn our attention to the \emph{discrete Caccioppoli inequality}. In a nutshell, it will allow us to bound an $H^1$-norm on a cluster $\ElementsB \subseteq \Elements$ by an $L^2$-norm on the slightly larger cluster $\inflateN{\ElementsB}{\delta}$. Obviously, this can be true only for a certain subspace $V \subseteq \Skp{1}{p}{\Elements}$. In our setting, this is the space of locally discrete harmonic functions $\SkpOHarm{1}{p}{\Elements}{\inflateN{\ElementsB}{\delta}}$ from \cref{Space_SkpOHarm}. We can interpret the discrete Caccioppoli inequality as an improved version of the inverse inequality from \cref{Inverse_inequality}, which bounds an $H^1$-seminorm by an $L^2$-norm, too. This time, however, the prefactor $\h{}$ of the $H^1$-seminorm can be increased to a (possibly much) bigger parameter $\delta \gg \h{}$.

\begin{lemma} \label{Space_SkpOHarm_Cacc}
Let $\ElementsB \subseteq \Elements$ and $\delta>0$ with $4\CShape^3 \h{\ElementsB} \leq \delta \cleq 1$. Then, for every $u \in \SkpOHarm{1}{p}{\Elements}{\inflateN{\ElementsB}{\delta}}$, there holds the \emph{discrete Caccioppoli inequality}
\begin{equation*}
\delta \seminorm[\Hk{1}{\ElementsB}]{u} \cleq \norm[\Lp{2}{\inflateN{\ElementsB}{\delta}}]{u}.
\end{equation*}
\end{lemma}

\begin{proof}
First off, an induction on $p \geq 1$ yields the following estimate: 
For every $\CutoffFc{}{} \in \Skp{0}{1}{\Elements}$, $u \in \Skp{0}{p}{\Elements}$ and $T \in \Elements$,
\begin{equation*}
\h{T}^{p+1} \seminorm[\Hk{p+1}{T}]{{\CutoffFc{}{}}^2 u} \cleq \h{T}^2 \seminorm[\Wkp{1}{\infty}{\Omega}]{\CutoffFc{}{}} ( \norm[\Lp{2}{T}]{u \gradN{\CutoffFc{}{}}} + \norm[\Lp{2}{T}]{\CutoffFc{}{} \gradN{u}} ).
\end{equation*}

In the base case $p=1$, the second-order derivatives in $\seminorm[\Hk{2}{T}]{{\CutoffFc{}{}}^2 u}$ can be computed explicitly. Since $\CutoffFc{}{}, u \in \Pp{1}{T}$, the terms containing $\DN{\alpha}{\CutoffFc{}{}}$ or $\DN{\alpha}{u}$ with $\abs{\alpha}=2$ are not present. In the induction step $p \mapsto p+1$, we estimate $\seminorm[\Hk{p+2}{T}]{{\CutoffFc{}{}}^2 u} \cleq \sum_i \seminorm[\Hk{p+1}{T}]{\CutoffFc{}{} (\partialN{i}{\CutoffFc{}{}}) u} + \seminorm[\Hk{p+1}{T}]{{\CutoffFc{}{}}^2 (\partialN{i}{u})}$. For the first summand, we use the inverse inequality \cref{Inverse_inequality} and get $\seminorm[\Hk{p+1}{T}]{\CutoffFc{}{} (\partialN{i}{\CutoffFc{}{}}) u} \cleq \h{T}^{-p} \seminorm[\Hk{1}{T}]{\CutoffFc{}{} (\partialN{i}{\CutoffFc{}{}}) u}$. Again, we can expand the derivatives explicitly and cancel all terms containing second order derivatives of $\CutoffFc{}{} \in \Pp{1}{T}$. The second summand is amenable to the induction hypothesis: $\seminorm[\Hk{p+1}{T}]{{\CutoffFc{}{}}^2 (\partialN{i}{u})} \cleq \h{T}^{1-p} \seminorm[\Wkp{1}{\infty}{\Omega}]{\CutoffFc{}{}} ( \norm[\Lp{2}{T}]{(\partialN{i}{u}) \gradN{\CutoffFc{}{}}} + \norm[\Lp{2}{T}]{\CutoffFc{}{} \grad{\partialN{i}{u}}} )$. These terms can be treated with the fact $\gradN{\CutoffFc{}{}} \equiv \mathrm{const}$, the identity $\CutoffFc{}{} \grad{\partialN{i}{u}} = \partial{i}{\CutoffFc{}{} \gradN{u}} - (\partialN{i}{\CutoffFc{}{}}) \gradN{u}$ and the inverse inequality \cref{Inverse_inequality} once again.

Now, let us turn our attention to the discrete Caccioppoli inequality itself. For this purpose, let $\ElementsB \subseteq \Elements$ and $\delta>0$ with $4\CShape^3 \h{\ElementsB} \leq \delta \cleq 1$. We denote by $\CutoffFc{}{} := \CutoffFc{\ElementsB}{\delta} \in \Skp{1}{1}{\Elements}$ the discrete cut-off function from \cref{Cut_off_fct_disc} and by $\fDef{I_{\Elements}^p}{\CkPw{0}{\Elements}}{\Skp{0}{p}{\Elements}}$ the Lagrange interpolation operator from \cref{App_op_Lagrange}. Furthermore, let $u \in \SkpOHarm{1}{p}{\Elements}{\inflateN{\ElementsB}{\delta}}$. The key step of the proof is to exploit the orthogonality $\bilinear[a]{u}{v} = 0$ for some carefully chosen test function $v \in \SkpO{1}{p}{\Elements}$ with $\supp[\Elements]{v} \subseteq \inflateN{\ElementsB}{\delta}$. From \cref{App_op_Lagrange_props} and \cref{Cut_off_fct_disc} we know that $v := I_{\Elements}^p({\CutoffFc{}{}}^2 u)$ satisfies both $v \in \SkpO{1}{p}{\Elements}$ and $\supp[\Elements]{v} \subseteq \supp[\Elements]{\CutoffFc{}{}} \subseteq \inflateN{\ElementsB}{\delta}$, i.e., we can use $v$ as said test function. This results in the following bound:
\begin{eqnarray*}
\bilinear[a]{u}{{\CutoffFc{}{}}^2 u} &=& \bilinear[a]{u}{(\identity - I_{\Elements}^p)({\CutoffFc{}{}}^2 u)} \\
&\stackrel{\cref{Bilinear_form}}{\cleq}& \sum_{T \in \inflateN{\ElementsB}{\delta}} \norm[\Hk{1}{T}]{u} \norm[\Hk{1}{T}]{(\identity - I_{\Elements}^p)({\CutoffFc{}{}}^2 u)} \\
&\stackrel{\cref{App_op_Lagrange_props}}{\cleq}& \sum_{T \in \inflateN{\ElementsB}{\delta}} \norm[\Hk{1}{T}]{u} \h{T}^p \seminorm[\Hk{p+1}{T}]{{\CutoffFc{}{}}^2 u} \\
&\cleq& \seminorm[\Wkp{1}{\infty}{\Omega}]{\CutoffFc{}{}} \sum_{T \in \inflateN{\ElementsB}{\delta}} \h{T} \norm[\Hk{1}{T}]{u} ( \norm[\Lp{2}{T}]{u \gradN{\CutoffFc{}{}}} + \norm[\Lp{2}{T}]{\CutoffFc{}{} \gradN{u}} ) \\
&\stackrel{\cref{Inverse_inequality}}{\cleq}& \seminorm[\Wkp{1}{\infty}{\Omega}]{\CutoffFc{}{}} \norm[\Lp{2}{\inflateN{\ElementsB}{\delta}}]{u} ( \norm[\Lp{2}{\Omega}]{u \gradN{\CutoffFc{}{}}} + \norm[\Lp{2}{\Omega}]{\CutoffFc{}{} \gradN{u}} ).
\end{eqnarray*}

On the other hand, using the coercivity of the PDE coefficient $a_1$ in the bilinear form $\bilinear[a]{\cdot}{\cdot}$, cf. \cref{SSec_Model_problem}, we can expand the term $\bilinear[a]{u}{{\CutoffFc{}{}}^2 u}$ and rearrange the summands:
\begin{eqnarray*}
\norm[\Lp{2}{\Omega}]{\CutoffFc{}{} \gradN{u}}^2 &\cleq& \skalar[\Lp{2}{\Omega}]{a_1 \CutoffFc{}{} \gradN{u}}{\CutoffFc{}{} \gradN{u}} \\
&\stackrel{\cref{Bilinear_form}}{=}& \bilinear[a]{u}{{\CutoffFc{}{}}^2 u} - 2\skalar[\Lp{2}{\Omega}]{a_1 \CutoffFc{}{} \gradN{u}}{u \gradN{\CutoffFc{}{}}} - \skalar[\Lp{2}{\Omega}]{a_2 \cdot \gradN{u}}{{\CutoffFc{}{}}^2 u} - \skalar[\Lp{2}{\Omega}]{a_3 u}{{\CutoffFc{}{}}^2 u} \\
&\cleq& \seminorm[\Wkp{1}{\infty}{\Omega}]{\CutoffFc{}{}} \norm[\Lp{2}{\inflateN{\ElementsB}{\delta}}]{u} ( \norm[\Lp{2}{\Omega}]{u \gradN{\CutoffFc{}{}}} + \norm[\Lp{2}{\Omega}]{\CutoffFc{}{} \gradN{u}} ) \\
&& \hspace{10em} + \norm[\Lp{2}{\Omega}]{\CutoffFc{}{} \gradN{u}} \norm[\Lp{2}{\Omega}]{u \gradN{\CutoffFc{}{}}} + \norm[\Lp{2}{\Omega}]{\CutoffFc{}{} \gradN{u}} \norm[\Lp{2}{\Omega}]{\CutoffFc{}{} u} + \norm[\Lp{2}{\Omega}]{\CutoffFc{}{} u}^2 \\
&\stackrel{\forall \epsilon>0}{\leq}& C_\epsilon\norm[\Wkp{1}{\infty}{\Omega}]{\CutoffFc{}{}}^2 \norm[\Lp{2}{\inflateN{\ElementsB}{\delta}}]{u}^2 + \epsilon \norm[\Lp{2}{\Omega}]{\CutoffFc{}{} \gradN{u}}^2.
\end{eqnarray*}

Finally, since the parameter $\epsilon>0$ from Young's inequality can be chosen arbitrarily small, we can absorb the last summand of the right-hand side in the left-hand side of the overall inequality. We end up with
\begin{equation*}
\seminorm[\Hk{1}{\ElementsB}]{u} \stackrel{\restrictN{\CutoffFc{}{}}{\ElementsB} \equiv 1}{\leq} \norm[\Lp{2}{\Omega}]{\CutoffFc{}{} \gradN{u}} \cleq \norm[\Wkp{1}{\infty}{\Omega}]{\CutoffFc{}{}} \norm[\Lp{2}{\inflateN{\ElementsB}{\delta}}]{u} \stackrel{\cref{Cut_off_fct_disc}}{\cleq} \frac{1}{\delta} \norm[\Lp{2}{\inflateN{\ElementsB}{\delta}}]{u}.
\end{equation*}

This concludes the proof of the discrete Caccioppoli inequality.
\end{proof}

\subsection{The single- and multi-step coarsening operators} \label{SSec_Coarse_ops}

In this subsection, we do the actual work in the construction of the subspace $V_{\ElementsB,\ElementsD,L} \subseteq \SkpO{1}{p}{\Elements}$ from \cref{SSec_Proof_overview}. We design the so called \emph{single-} and \emph{multi-step coarsening operators}. For given $\ElementsB \subseteq \Elements$, $\delta>0$ and $u \in \SkpOHarm{1}{p}{\Elements}{\inflateN{\ElementsB}{\delta}}$, the single-step coarsening operator $\CoarseningOp{\ElementsB}{\delta}$ produces a ``coarse'' approximation $\CoarseningOp{\ElementsB}{\delta} u \in \SkpOHarm{1}{p}{\Elements}{\ElementsB}$ with an error $\norm[\Lp{2}{\ElementsB}]{u - \CoarseningOp{\ElementsB}{\delta} u} \leq 2^{-1} \norm[\Lp{2}{\inflateN{\ElementsB}{\delta}}]{u}$. The prefactor $2^{-1} \in (0,1)$ is essential, as it produces an exponential factor $2^{-L}$ when $L \in \N$ single-step coarsening operators are combined in a specific manner. This is precisely the idea behind the multi-step coarsening operator $\CoarseningOp{\ElementsB}{\delta,L}$. Given a function $u \in \SkpOHarm{1}{p}{\Elements}{\inflateN{\ElementsB}{\delta L}}$, it produces a ``coarse'' approximation $\CoarseningOp{\ElementsB}{\delta,L} u \in \SkpOHarm{1}{p}{\Elements}{\ElementsB}$ with an error $\norm[\Lp{2}{\ElementsB}]{u - \CoarseningOp{\ElementsB}{\delta,L} u} \leq 2^{-L} \norm[\Lp{2}{\inflateN{\ElementsB}{\delta L}}]{u}$.

As our construction of the single-step coarsening operator in \cref{Coarsening_op_single} is quite technical, we would like to reveal the underlying deas first: Assume for a moment that $\Elements$ is uniform, i.e. $\h{\Elements} \ceq \hMin{\Elements}$. Then, a function $u \in \SkpOHarm{1}{p}{\Elements}{\inflateN{\ElementsB}{\delta}}$ is described by up to $\dimN{\Skp{0}{p}{\Elements}} \ceq \cardN{\Elements} \ceq \h{\Elements}^{-d}$ degrees of freedom. In order to reduce this number, we could approximate $u \approx \Pi_{\ElementsS}^p u \in \Skp{0}{p}{\ElementsS}$, where $\ElementsS \subseteq \Pow{\Omega}$ is a second uniform mesh and where $\fDef{\Pi_{\ElementsS}^p}{\Lp{2}{\Omega}}{\Skp{0}{p}{\ElementsS}}$ is some kind of approximation operator. As long as $\ElementsS$ is coarser than $\Elements$, i.e. $\h{\ElementsS} \cgeq \h{\Elements}$, this provides a reduction of the dimension. On the other hand, the typical error bound $\norm[\Lp{2}{\Omega}]{u - \Pi_{\ElementsS}^p u} \cleq H \seminorm[\Hk{1}{\Omega}]{u}$ involves an $H^1$-norm on the right-hand side. In order to get rid of the $H^1$-norm, we want to apply the discrete Caccioppoli inequality, \cref{Space_SkpOHarm_Cacc}. For this to work, however, we first need to reduce the global quantity $H \seminorm[\Hk{1}{\Omega}]{u}$ to the local quantity $H \seminorm[\Hk{1}{\ElementsB}]{u}$. This can be done using the discrete cut-off operator $\CutoffOp{\ElementsB}{\delta}$ from \cref{Cut_off_op_disc}. Finally, the combined operator $\fDef{\Pi_{\ElementsS}^p \CutoffOp{\ElementsB}{\delta}}{\SkpOHarm{1}{p}{\Elements}{\inflateN{\ElementsB}{\delta}}}{\Skp{0}{p}{\ElementsS}}$ only lacks one more thing: It does not necessarily map into the space $\SkpOHarm{1}{p}{\Elements}{\ElementsB}$, which is a critical requirement, because we want to iterate the argument by plugging the remainder $\tilde{u} := u - \CoarseningOp{\ElementsB}{\delta} u$ of one single-step coarsening operator  into another one. Thankfully, we can simply append the orthogonal projection $\fDef{P_{\ElementsB}}{\Lp{2}{\Omega}}{\SkpOHarm{1}{p}{\Elements}{\ElementsB}}$ without losing any of the aforementioned properties.

In the next lemma we provide a construction for the second, coarser mesh $\ElementsS \subseteq \Pow{\Omega}$:

\begin{lemma} \label{Uniform_mesh_at_h}
Let $\ElementsS_0 \subseteq \Pow{\Omega}$ be an arbitrary mesh and $(\ElementsS_\ell)_{\ell \in \N_0}$ be the corresponding sequence of uniform refinements. For every $H>0$, there exists an $\ElementsS \in (\ElementsS_\ell)_{\ell \in \N_0}$ with $\CShape(\ElementsS) = C(\Omega)$ and $C(\Omega) H \leq \hMin{\ElementsS} \leq \h{\ElementsS} \leq H$. In particular, $\ElementsS$ is uniform in the sense of \cref{Mesh_uniform}.
\end{lemma}

\begin{proof}
There hold the relations $\h{\ElementsS_\ell} = 2^{-\ell} \h{\ElementsS_0}$ and $\hMin{\ElementsS_\ell} = 2^{-\ell} \hMin{\ElementsS_0}$. For any given $H>0$, we choose the mesh $\ElementsS := \ElementsS_L$, where $L \in \N_0$ is the minimal level satisfying $\h{\ElementsS_L} \leq H$. In particular, there also holds the lower bound $H < \h{\ElementsS_{L-1}} = 2^{-(L-1)} \h{\ElementsS_0} = 2 \h{\ElementsS_0} \hMin{\ElementsS_0}^{-1} \hMin{\ElementsS_L} = C(\Omega) \hMin{\ElementsS_L}$.
\end{proof}

The additional mesh $\ElementsS \subseteq \Pow{\Omega}$ does not need to be aligned with the original mesh $\Elements \subseteq \Pow{\Omega}$ at all. The output of the cut-off operator $\CutoffOp{\ElementsB}{\epsilon}$ is just an element of $\SkpO{1}{p}{\Elements} \subseteq \Hk{1}{\Omega}$, so we need an operator $\fDef{\Pi_{\ElementsS}}{\Hk{1}{\Omega}}{\Skp{0}{q}{\ElementsS}}$ for some $q \geq 0$. Also, in the case $\ElementsS = \Elements$ the operator should act like a projection on functions from $\SkpO{1}{p}{\Elements}$. The simplest solution for these demands is the \emph{piecewise orthogonal projection}.

\begin{definition} \label{App_op_L2_proj}
Let $\ElementsS \subseteq \Pow{\Omega}$ be a mesh, $p \geq 0$ and $\fDef{\hat{\Pi}^p}{\Lp{2}{\refElement}}{\Pp{p}{\refElement}}$ the orthogonal projection on the reference element $\refElement$. The \emph{piecewise orthogonal projection} $\fDef{\Pi_{\ElementsS}^p}{\Lp{2}{\Omega}}{\Skp{0}{p}{\ElementsS}}$ is defined in a piecewise manner: For every $v \in \Lp{2}{\Omega}$ and every $S \in \ElementsS$ we set
\begin{equation*}
\restrict{\Pi_{\ElementsS}^p v}{S} := \hat{\Pi}^p(v \circ F_S) \circ F_S^{-1}.
\end{equation*}
\end{definition}

In fact, $\Pi_{\ElementsS}^p$ coincides with the (global) orthogonal projection from $\Lp{2}{\Omega}$ onto the closed subspace $\Skp{0}{p}{\ElementsS}$. The piecewise approach, however, results in desirable \emph{local} properties and bounds.

\begin{lemma} \label{App_op_L2_proj_props}
The linear operator $\Pi_{\ElementsS}^p$ has a local projection property: For every cluster $\ElementsB \subseteq \ElementsS$ and every function $v \in \Lp{2}{\Omega}$ with $v \in \Skp{0}{p}{\ElementsB}$, there holds $\restrict{\Pi_{\ElementsS}^p v}{\ElementsB} = \restrictN{v}{\ElementsB}$. Furthermore, $\Pi_{\ElementsS}^p$ preserves supports: For every $v \in \Lp{2}{\Omega}$, we have $\supp[\ElementsS]{\Pi_{\ElementsS}^p v} \subseteq \supp[\ElementsS]{v}$. Finally, for every $k \in \set{0,\dots,p+1}$, there hold the stability and error estimates
\begin{equation*}
\begin{array}{lrcl}
\forall v \in \HkPw{k}{\ElementsS}: \forall S \in \ElementsS: \quad & \sum_{\ell=0}^{k} \h{S}^{\ell} \seminorm[\Hk{\ell}{S}]{\Pi_{\ElementsS}^p v} &\cleq& \sum_{\ell=0}^{k} \h{S}^{\ell} \seminorm[\Hk{\ell}{S}]{v}, \\
 & \sum_{\ell=0}^{k} \h{S}^{\ell} \seminorm[\Hk{\ell}{S}]{(\identity - \Pi_{\ElementsS}^p)(v)} &\cleq& \h{S}^k \seminorm[\Hk{k}{S}]{v}.
\end{array}
\end{equation*}
\end{lemma}

Now, we have all the ingredients for the construction of the \emph{single-step coarsening operator}.

\begin{theorem} \label{Coarsening_op_single}
Let $\Elements \subseteq \Pow{\Omega}$ be a mesh of locally bounded cardinality. Furthermore, let $\ElementsB \subseteq \Elements$ and $\delta>0$ with $\delta \cleq 1$. Then there exists a linear \emph{single-step coarsening operator}
\begin{equation*}
\fDef{\CoarseningOp{\ElementsB}{\delta}}{\SkpOHarm{1}{p}{\Elements}{\inflateN{\ElementsB}{\delta}}}{\SkpOHarm{1}{p}{\Elements}{\ElementsB}}
\end{equation*}
of rank
\begin{equation*}
\rank{\CoarseningOp{\ElementsB}{\delta}} \cleq \( 1 + \frac{\diam[\Elements]{\ElementsB}}{\delta} \)^{d \CCard}
\end{equation*}
that satisfies the following approximation property: For every $u \in \SkpOHarm{1}{p}{\Elements}{\inflateN{\ElementsB}{\delta}}$,
\begin{equation*}
\norm[\Lp{2}{\ElementsB}]{u - \CoarseningOp{\ElementsB}{\delta} u} \leq \frac{1}{2} \norm[\Lp{2}{\inflateN{\ElementsB}{\delta}}]{u}.
\end{equation*}
\end{theorem}

\begin{proof}
Let $\ElementsB \subseteq \Elements$ and $\delta>0$ with $\delta \cleq 1$. For the construction of $\CoarseningOp{\ElementsB}{\delta}$ we need three operators: First, we use the discrete cut-off operator $\fDef{\CutoffOp{\ElementsB}{\epsilon}}{\Skp{1}{p}{\Elements}}{\Skp{1}{p}{\Elements}}$ from \cref{Cut_off_op_disc} with some carefully chosen parameter $\epsilon>0$. Second, we apply the piecewise orthogonal projection $\fDef{\Pi_{\ElementsS}^p}{\Lp{2}{\Omega}}{\Skp{0}{p}{\ElementsS}}$ from \cref{App_op_L2_proj} on some suitable mesh $\ElementsS \subseteq \Pow{\Omega}$. Third, the result is mapped back into the space $\SkpOHarm{1}{p}{\Elements}{\ElementsB}$ via the orthogonal projection $\fDef{P_{\ElementsB}}{\Lp{2}{\Omega}}{\SkpOHarm{1}{p}{\Elements}{\ElementsB}}$.

For the precise choice of $\epsilon$ and $\ElementsS$ we have to distinguish between two cases: In the more involved case $\delta \geq 20\CShape^7 \h{\ElementsB}$ we choose $\epsilon := \delta/(5\CShape^4) \geq 4\CShape^3 \h{\ElementsB}$ and use the uniform mesh $\ElementsS \subseteq \Pow{\Omega}$ from \cref{Uniform_mesh_at_h} with $\h{\ElementsS} \ceq \hMin{\ElementsS} \ceq H$, where the parameter $H>0$ will be specified during the proof. In the degenerate case $\delta < 20\CShape^7 \h{\ElementsB}$ we set $\epsilon := 4\CShape^3 \h{\ElementsB}$ and use the mesh $\ElementsS := \Elements$ itself.

We define the asserted operator as
\begin{equation*}
\fDef{\CoarseningOp{\ElementsB}{\delta} := P_{\ElementsB} \Pi_{\ElementsS}^p \CutoffOp{\ElementsB}{\epsilon}}{\SkpOHarm{1}{p}{\Elements}{\inflateN{\ElementsB}{\delta}}}{\SkpOHarm{1}{p}{\Elements}{\ElementsB}}.
\end{equation*}

\emph{The case $\delta \geq 20\CShape^7 \h{\ElementsB}$:} Let $u \in \SkpOHarm{1}{p}{\Elements}{\inflateN{\ElementsB}{\delta}}$. From \cref{Mesh_Metric_props} we know that the parameter $\alpha := 4\CShape^4 \epsilon$ satisfies $4\CShape^3 \h{\inflateN{\ElementsB}{\epsilon}} \leq \alpha \cleq 1$. In particular, we can apply the discrete Caccioppoli inequality to the set $\inflateN{\ElementsB}{\epsilon}$ and the parameter $\alpha$. Since $\delta \eqsim \alpha$, this gives the stability estimate for the cut-off operator $\CutoffOp{\ElementsB}{\epsilon}$
\begin{equation*}
\sum_{\ell=0}^{1} \delta^\ell \seminorm[\Hk{\ell}{\Omega}]{\CutoffOp{\ElementsB}{\epsilon} u} \stackrel{\cref{Cut_off_op_disc_props}}{\cleq} \sum_{\ell=0}^{1} \alpha^\ell \seminorm[\Hk{\ell}{\inflateN{\ElementsB}{\epsilon}}]{u} \stackrel{\cref{Space_SkpOHarm_Cacc}}{\cleq} \norm[\Lp{2}{\inflateN{\ElementsB}{\epsilon+\alpha}}]{u} \stackrel{\epsilon+\alpha \leq \delta}{\leq} \norm[\Lp{2}{\inflateN{\ElementsB}{\delta}}]{u}.
\end{equation*}

From \cref{Space_SkpOHarm_props} and \cref{Cut_off_op_disc_props} we know that $\CutoffOp{\ElementsB}{\epsilon} u \in \SkpOHarm{1}{p}{\Elements}{\ElementsB}$, hence $P_{\ElementsB} \CutoffOp{\ElementsB}{\epsilon} u = \CutoffOp{\ElementsB}{\epsilon} u$. We conclude $\restrictN{u}{\ElementsB} = \restrict{\CutoffOp{\ElementsB}{\epsilon} u}{\ElementsB} = \restrict{P_{\ElementsB} \CutoffOp{\ElementsB}{\epsilon} u}{\ElementsB}$ and thus
\begin{equation*}
\begin{array}{rclcl}
\norm[\Lp{2}{\ElementsB}]{u - \CoarseningOp{\ElementsB}{\delta} u} &=& \norm[\Lp{2}{\ElementsB}]{P_{\ElementsB} \CutoffOp{\ElementsB}{\epsilon} u - P_{\ElementsB} \Pi_{\ElementsS}^p \CutoffOp{\ElementsB}{\epsilon} u} &\leq& \norm[\Lp{2}{\Omega}]{P_{\ElementsB} (\identity - \Pi_{\ElementsS}^p)(\CutoffOp{\ElementsB}{\epsilon} u)} \\
&\leq& \norm[\Lp{2}{\Omega}]{(\identity - \Pi_{\ElementsS}^p)(\CutoffOp{\ElementsB}{\epsilon} u)} &\stackrel{\cref{App_op_L2_proj_props}}{\cleq}& H \seminorm[\Hk{1}{\Omega}]{\CutoffOp{\ElementsB}{\epsilon} u} \\
&\cleq& \frac{H}{\delta} \norm[\Lp{2}{\inflateN{\ElementsB}{\delta}}]{u}.
\end{array}
\end{equation*}

In particular, we can choose $H \ceq \delta > 0$ small enough to establish the asserted error bound.

\emph{The case $\delta < 20\CShape^7 \h{\ElementsB}$:} Again let $u \in \SkpOHarm{1}{p}{\Elements}{\inflateN{\ElementsB}{\delta}}$. Exploiting $\ElementsS = \Elements$ and \cref{Cut_off_op_disc_props}, the operator $\CoarseningOp{\ElementsB}{\delta}$ reduces to $\CoarseningOp{\ElementsB}{\delta} u = P_{\ElementsB} \Pi_{\Elements}^p \CutoffOp{\ElementsB}{\epsilon} u = P_{\ElementsB} \CutoffOp{\ElementsB}{\epsilon} u = \CutoffOp{\ElementsB}{\epsilon} u$. Consequently, the error bound becomes trivial:
\begin{equation*}
\norm[\Lp{2}{\ElementsB}]{u - \CoarseningOp{\ElementsB}{\delta} u} = \norm[\Lp{2}{\ElementsB}]{u - \CutoffOp{\ElementsB}{\epsilon} u} = \norm[\Lp{2}{\ElementsB}]{u - u} = 0.
\end{equation*}

To find a good upper bound for the rank of $\CoarseningOp{\ElementsB}{\delta}$, the locally bounded cardinality of $\ElementsS$ is crucial. In the case $\delta \geq 20\CShape^7 \h{\ElementsB}$ the mesh $\ElementsS$ is uniform and thus of locally bounded cardinality (cf. \cref{Mesh_uniform_card}). In the case $\delta < 20\CShape^7 \h{\ElementsB}$ we chose $\ElementsS = \Elements$, which has locally bounded cardinality by assumption.

Next, we abbreviate $B := \bigcup\inflateN{\ElementsB}{\epsilon} \subseteq \R^d$ and compute a common lower bound for $\h{\Patch{\ElementsS}{B}}$: In the case $\delta \geq 20\CShape^7 \h{\ElementsB}$ we have $\h{\ElementsB} + \epsilon + \delta \cleq \delta \ceq H \ceq \hMin{\ElementsS} \leq \h{\Patch{\ElementsS}{B}}$ and in the case $\delta < 20\CShape^7 \h{\ElementsB}$ we get $\h{\ElementsB} + \epsilon + \delta \cleq \h{\ElementsB} \leq \h{\Patch{\Elements}{\inflateN{\ElementsB}{\epsilon}}} = \h{\Patch{\ElementsS}{B}}$ as well.

Now, for every $u \in \SkpOHarm{1}{p}{\Elements}{\inflateN{\ElementsB}{\delta}}$ we know from \cref{App_op_L2_proj_props} and \cref{Cut_off_op_disc_props} that $\supp[\ElementsS]{\Pi_{\ElementsS}^p \CutoffOp{\ElementsB}{\epsilon} u} \subseteq \supp[\ElementsS]{\CutoffOp{\ElementsB}{\epsilon} u} \subseteq \Patch{\ElementsS}{B}$. This results in the estimate
\begin{equation*}
\begin{array}{rclcl}
\rank{\CoarseningOp{\ElementsB}{\delta}} &\leq& \dimN{\Set{v \in \Skp{0}{p}{\ElementsS}}{\supp[\ElementsS]{v} \subseteq \Patch{\ElementsS}{B}}} &\ceq& \cardN{\Patch{\ElementsS}{B}} \\
&\stackrel{\cref{Mesh_loc_bd_card}}{\cleq}& ( 1 + \h{\Patch{\ElementsS}{B}}^{-1} \diam[\ElementsS]{\Patch{\ElementsS}{B}} )^{d\CCard} &\stackrel{\cref{Mesh_Metric_props}}{\cleq}& ( 1 + \h{\Patch{\ElementsS}{B}}^{-1}(\diam[\Elements]{\ElementsB} + \h{\ElementsB} + \epsilon) )^{d\CCard} \\
&\stackrel{\h{\ElementsB} + \epsilon + \delta \cleq \h{\Patch{\ElementsS}{B}}}{\cleq}& ( 1 + \delta^{-1} \diam[\Elements]{\ElementsB} )^{d\CCard},
\end{array}
\end{equation*}
which finishes the proof.
\end{proof}

With the single-step coarsening operator at hand, we can iterate to obtain exponential convergence.

\begin{theorem} \label{Coarsening_op_multi}
Let $\Elements \subseteq \Pow{\Omega}$ be a mesh of locally bounded cardinality. Furthermore, let $\ElementsB \subseteq \Elements$ and $\delta>0$ with $\delta \cleq 1$. Then, for every $L \in \N$, there exists a linear \emph{multi-step coarsening operator}
\begin{equation*}
\fDef{\CoarseningOp{\ElementsB}{\delta,L}}{\SkpOHarm{1}{p}{\Elements}{\inflateN{\ElementsB}{\delta L}}}{\SkpOHarm{1}{p}{\Elements}{\ElementsB}}
\end{equation*}
of rank
\begin{equation*}
\rank{\CoarseningOp{\ElementsB}{\delta,L}} \cleq \( L + \frac{\diam[\Elements]{\ElementsB}}{\delta} \)^{d\CCard+1}
\end{equation*}
that satisfies the following approximation property: For every $u \in \SkpOHarm{1}{p}{\Elements}{\inflateN{\ElementsB}{\delta L}}$, there holds
\begin{equation*}
\norm[\Lp{2}{\ElementsB}]{u - \CoarseningOp{\ElementsB}{\delta,L} u} \leq 2^{-L} \norm[\Lp{2}{\inflateN{\ElementsB}{\delta L}}]{u}.
\end{equation*}
\end{theorem}

\begin{proof}
Let $\ElementsB \subseteq \Elements$ and $\delta>0$ with $\delta \cleq 1$ as well as $L \in \N$. We define a sequence of nested element sets $\ElementsB \subseteq \ElementsB_0 \subseteq \dots \subseteq \ElementsB_L \subseteq \inflateN{\ElementsB}{\delta L}$ inductively by $\ElementsB_0 := \ElementsB$ and $\ElementsB_{\ell+1} := \inflate{\ElementsB_{\ell}}{\delta}$. Using the corresponding single-step coarsening operators $\fDef{\CoarseningOp{\ell}{} := \CoarseningOp{\ElementsB_{\ell}}{\delta}}{\SkpOHarm{1}{p}{\Elements}{\ElementsB_{\ell+1}}}{\SkpOHarm{1}{p}{\Elements}{\ElementsB_{\ell}}}$ from \cref{Coarsening_op_single}, we make the following definition:
\begin{equation*}
\forall u \in \SkpOHarm{1}{p}{\Elements}{\inflateN{\ElementsB}{\delta L}}: \quad \quad \CoarseningOp{\ElementsB}{\delta,L} u := u - (\identity - \CoarseningOp{0}{}) \circ \dots \circ (\identity - \CoarseningOp{L-1}{})(u) \in \SkpOHarm{1}{p}{\Elements}{\ElementsB}.
\end{equation*}

Using the alternative representation $\CoarseningOp{\ElementsB}{\delta,L} u = -\sum_{\pi \in \set{0,1}^L\backslash\set{0}} (-\CoarseningOp{0}{})^{(\pi_0)} \circ \dots \circ (-\CoarseningOp{L-1}{})^{(\pi_{L-1})} (u)$, we infer
\begin{equation*}
\begin{array}{rclcl}
\rank{\CoarseningOp{\ElementsB}{\delta,L}} &\leq& \sum_{\ell=0}^{L-1} \rank{\CoarseningOp{\ell}{}} &\stackrel{\cref{Coarsening_op_single}}{\cleq}& \sum_{\ell=0}^{L-1} ( 1 + \delta^{-1} \diam[\Elements]{\ElementsB_{\ell}} )^{d \CCard} \\
&\stackrel{\cref{Mesh_Metric_props}}{\cleq}& \sum_{\ell=0}^{L-1} ( 1 + \delta^{-1} \diam[\Elements]{\ElementsB} + \ell )^{d \CCard} &\leq& ( L + \delta^{-1} \diam[\Elements]{\ElementsB} )^{d\CCard+1}.
\end{array}
\end{equation*}

Finally, the definition of $\CoarseningOp{\ElementsB}{\delta,L}$ was such that the error bound becomes elementary: For every $u \in \SkpOHarm{1}{p}{\Elements}{\inflateN{\ElementsB}{\delta L}}$, iteration of \cref{Coarsening_op_single} gives
\begin{equation*}
\norm[\Lp{2}{\ElementsB}]{u - \CoarseningOp{\ElementsB}{\delta,L} u} = \norm[\Lp{2}{\ElementsB_0}]{(\identity - \CoarseningOp{0}{}) \circ \dots \circ (\identity - \CoarseningOp{L-1}{})(u)} \leq 2^{-L} \norm[\Lp{2}{\inflateN{\ElementsB}{\delta L}}]{u}.
\end{equation*}
\end{proof}

\subsection{Putting everything together} \label{SSec_Put_together}

We can finally answer the question of how to find the subspace $V_{\ElementsB,\ElementsD,L} \subseteq \Lp{2}{\Omega}$ from \cref{SSec_Proof_overview}. After that, the proof of \cref{System_matrix_HMatrix_approx} is just a matter of putting everything together.

\begin{theorem} \label{Space_VBDL}
Let $\Elements \subseteq \Pow{\Omega}$ be a mesh of locally bounded cardinality and $\ElementsB,\ElementsD \subseteq \Elements$ clusters satisfying
\begin{equation*}
0 < \diam[\Elements]{\ElementsB} \leq \CAdm \dist[\Elements]{\ElementsB}{\ElementsD}.
\end{equation*}

Then, for every $L \in \N$, there exists a subspace
\begin{equation*}
V_{\ElementsB,\ElementsD,L} \subseteq \SkpO{1}{p}{\Elements}
\end{equation*}
of dimension
\begin{equation*}
\dim{V_{\ElementsB,\ElementsD,L}} \cleq L^{d\CCard+1}
\end{equation*}
that satisfies the following approximation property: For every $f \in \Lp{2}{\Omega}$ with $\supp[\Elements]{f} \subseteq \ElementsD$, 
\begin{equation*}
\inf_{v \in V_{\ElementsB,\ElementsD,L}} \norm[\Lp{2}{\ElementsB}]{\SolutionOp{\Elements} f - v} \cleq 2^{-L} \norm[\Lp{2}{\ElementsD}]{f}.
\end{equation*}
\end{theorem}

\begin{proof}
Let $\ElementsB,\ElementsD \subseteq \Elements$ with $0 < \diam[\Elements]{\ElementsB} \leq \CAdm \dist[\Elements]{\ElementsB}{\ElementsD}$. For every given $L \in \N$, we make the choice $\delta := \diam[\Elements]{\ElementsB}/(2\CAdm L) > 0$ and use the space
\begin{equation*}
V_{\ElementsB,\ElementsD,L} := \ran{\CoarseningOp{\ElementsB}{\delta,L}} \subseteq \SkpO{1}{p}{\Elements}.
\end{equation*}
Here, $\fDef{\CoarseningOp{\ElementsB}{\delta,L}}{\SkpOHarm{1}{p}{\Elements}{\inflateN{\ElementsB}{\delta L}}}{\SkpOHarm{1}{p}{\Elements}{\ElementsB}}$ is the multi-step coarsening operator from \cref{Coarsening_op_multi}.

Using \cref{Coarsening_op_multi} and the definition of $\delta$, we can bound the dimension by
\begin{equation*}
\dim{V_{\ElementsB,\ElementsD,L}} = \rank{\CoarseningOp{\ElementsB}{\delta,L}} \cleq \( L + \frac{\diam[\Elements]{\ElementsB}}{\delta} \)^{d\CCard+1} \cleq L^{d\CCard+1}.
\end{equation*}

Finally, let $f \in \Lp{2}{\Omega}$ with $\supp[\Elements]{f} \subseteq \ElementsD$. By definition of $\inflateN{\ElementsB}{\delta L}$ and $\dist[\Elements]{\inflateN{\ElementsB}{\delta L}}{\ElementsD}$, there exist elements $B \in \ElementsB$, $C \in \inflateN{\ElementsB}{\delta L}$, $D \in \ElementsD$ such that $\dist[\Elements]{B}{C} \leq \delta L$ and $\dist[\Elements]{\inflateN{\ElementsB}{\delta L}}{\ElementsD} = \dist[\Elements]{C}{D}$. Using the triangle inequality of the mesh metric $\dist[\Elements]{\cdot}{\cdot}$, we conclude $\dist[\Elements]{\ElementsB}{\ElementsD} \leq \dist[\Elements]{B}{D} \leq \dist[\Elements]{B}{C} + \dist[\Elements]{C}{D} \leq \delta L + \dist[\Elements]{\inflateN{\ElementsB}{\delta L}}{\ElementsD}$. Now, exploiting the definition of $\delta$ and the assumptions on $\ElementsB,\ElementsD$, we obtain
\begin{equation*}
\dist[\Elements]{\inflateN{\ElementsB}{\delta L}}{\ElementsD} \geq \dist[\Elements]{\ElementsB}{\ElementsD} - \delta L = \dist[\Elements]{\ElementsB}{\ElementsD} - \frac{\diam[\Elements]{\ElementsB}}{2\CAdm} \geq \frac{\diam[\Elements]{\ElementsB}}{2\CAdm} > 0.
\end{equation*}

Then, \cref{Space_SkpOHarm_props} implies $\SolutionOp{\Elements} f \in \SkpOHarm{1}{p}{\Elements}{\inflateN{\ElementsB}{\delta L}}$ and ultimately
\begin{equation*}
\inf_{v \in V_{\ElementsB,\ElementsD,L}} \norm[\Lp{2}{\ElementsB}]{\SolutionOp{\Elements} f - v} \leq \norm[\Lp{2}{\ElementsB}]{\SolutionOp{\Elements} f - \CoarseningOp{\ElementsB}{\delta,L}(\SolutionOp{\Elements} f)} \stackrel{\cref{Coarsening_op_multi}}{\leq} 2^{-L} \norm[\Lp{2}{\inflateN{\ElementsB}{\delta L}}]{\SolutionOp{\Elements} f} \stackrel{\cref{Sol_op_disc}}{\cleq} 2^{-L} \norm[\Lp{2}{\ElementsD}]{f}.
\end{equation*}
\end{proof}

We close this section with the proof of \cref{System_matrix_HMatrix_approx}.

\begin{proof}
Let $\mvemph{A} \in \R^{N \times N}$ be the matrix from \cref{System_matrix} and $r \in \N$ a given block rank bound. We define the asserted $\mathcal{H}$-matrix approximant $\mvemph{B} \in \R^{N \times N}$ to $\mvemph{A}^{-1}$ in a block-wise fashion:

First, for every admissible block $(I,J) \in \BPartAdm$, we denote the corresponding index patches by $\ElementsB := \Patch{\Elements}{I} \subseteq \Elements$ and $\ElementsD := \Patch{\Elements}{J} \subseteq \Elements$. From \cref{Block_partition} we know that $0 < \diam[\Elements]{\ElementsB} \leq \CAdm \dist[\Elements]{\ElementsB}{\ElementsD}$. Furthermore, let $C>0$ be the constant from the dimension bound in \cref{Space_VBDL}. We set $\CExp := (1/C)^{1/(d\CCard+1)} \ln(2) > 0$ and $L := \floor{(r/C)^{1/(d\CCard+1)}} \in \N$. Then, \cref{Space_VBDL} provides a subspace $V_{\ElementsB,\ElementsD,L} \subseteq \SkpO{1}{p}{\Elements} \subseteq \Lp{2}{\Omega}$. We apply \cref{System_matrix_block_approx} to the subspace $V_{\ElementsB,\ElementsD,L} \subseteq \Lp{2}{\Omega}$ and get matrices $\mvemph{X}_{I,J}^r \in \R^{I \times \tilde{r}}$ and $\mvemph{Y}_{I,J}^r \in \R^{J \times \tilde{r}}$ of size $\tilde{r} \leq \dim{V_{\ElementsB,\ElementsD,L}}$. We set
\begin{equation*}
\restrictN{\mvemph{B}}{I \times J} := \mvemph{X}_{I,J}^r (\mvemph{Y}_{I,J}^r)^T.
\end{equation*}

Second, for every small block $(I,J) \in \BPartSmall$, we make the trivial choice
\begin{equation*}
\restrictN{\mvemph{B}}{I \times J} := \restrictN{\mvemph{A}^{-1}}{I \times J}.
\end{equation*}

By \cref{H_matrices}, we have $\mvemph{B} \in \HMatrices{\BPart}{\tilde{r}}$ with a block rank bound
\begin{equation*}
\tilde{r} \leq \dim{V_{\ElementsB,\ElementsD,L}} \stackrel{\cref{Space_VBDL}}{\leq} C L^{d\CCard+1} \leq r.
\end{equation*}

For the error we get
\begin{eqnarray*}
\norm[2]{\mvemph{A}^{-1} - \mvemph{B}} &\stackrel{\cref{Block_partition_matrix_norm}}{\leq}& N^2 \* \max_{(I,J) \in \BPartAdm} \norm[2]{\restrictN{\mvemph{A}^{-1}}{I \times J} - \mvemph{X}_{I,J}^r (\mvemph{Y}_{I,J}^r)^T} \\
&\stackrel{\cref{System_matrix_block_approx}}{\leq}& N^2 \norm{\Lambda}^2 \* \max_{\substack{\ElementsB, \ElementsD \subseteq \Elements \\ \text{admissible}}} \sup_{\substack{f \in \Lp{2}{\Omega}: \\ \supp[\Elements]{f} \subseteq \ElementsD}} \norm[\Lp{2}{\ElementsD}]{f}^{-1} \* \inf_{v \in V_{\ElementsB,\ElementsD,L}} \norm[\Lp{2}{\ElementsB}]{\SolutionOp{\Elements} f - v} \\
&\stackrel{\cref{Space_VBDL}}{\cleq}& N^2 \norm{\Lambda}^2 2^{-L} \\
&\cleq& N^2 \norm{\Lambda}^2 \exp(-\CExp r^{1/(d\CCard + 1)}).
\end{eqnarray*}

Finally, it only remains to bound the norm of $\Lambda$:
\begin{equation*}
\norm{\Lambda}^2 \stackrel{\cref{Basis_fcts}}{\cleq} \hMin{\Elements}^{-d} \stackrel{\cref{Mesh_loc_bd_card}}{\cleq} \h{\Elements}^{-d\CCard} \stackrel{\cref{Mesh_card}}{\cleq} \cardN{\Elements}^{\CCard} \ceq (\dimN{\SkpO{1}{p}{\Elements}})^{\CCard} = N^{\CCard}.
\end{equation*}

This concludes the proof of the main result, \cref{System_matrix_HMatrix_approx}.
\end{proof}

\section{Numerical results} \label{Sec_Numerical_results}

In this subsection, we illustrate the validity of \cref{System_matrix_HMatrix_approx} by means of a numerical example:

For the geometry we choose the \emph{L-shaped domain} $\Omega := ((0,1) \times (0,1)) \backslash ([1/2,1] \times [1/2,1]) \subseteq \R^2$ in two space dimensions. The PDE coefficients for the model problem from \cref{SSec_Model_problem} are given by $a_1(x) = (\begin{smallmatrix} 10 & -1 \\ -1 & 1 \end{smallmatrix})$, $a_2(x) := (\begin{smallmatrix} 10 x_2 \\ 0 \end{smallmatrix})$ and $a_3(x) := 1$. The mesh $\Elements$ is \emph{graded} in the sense of \cref{Mesh_graded} towards $\Gamma := \set{(1/2,1/2)}$ with exponent $\alpha := 5$ and the coarse mesh width $H := 0.0095$. We use the spline space $\SkpO{1}{1}{\Elements}$ ($p=1$, globally continuous, piecewise linear) and the well-known basis of \emph{hat-functions} $\set{\phi_1,\dots,\phi_N} \subseteq \SkpO{1}{1}{\Elements}$. The block partition $\BPart$ is constructed from a \emph{geometrically balanced cluster tree $\Tree{N}$} as suggested in \cite{Grasedyck_Clustering}. We choose the parameters $\CAdm := 2$ and $\CSmall := 25$ (cf. \cref{Block_partition}). For the rank bound we choose the range $r \in \set{1,\dots,50}$.

Unfortunately, the $\mathcal{H}$-matrix approximant $\mvemph{B} \in \R^{N \times N}$ from our proof is only a theoretical tool and inaccessible for an implementation in a computer system. Hence, we revert to a \emph{block-wise singular values decomposition}: First, we compute the exact inverse $\mvemph{A}^{-1} \in \R^{N \times N}$ explicitly. Then, for every admissible block $(I,J) \in \BPartAdm$, we perform the singular values decomposition $\restrictN{\mvemph{A}^{-1}}{I \times J} = \mvemph{U} \mvemph{\Sigma} \mvemph{V}^T \in \R^{I \times J}$. Here, $\mvemph{U} \in \R^{I \times I}, \mvemph{V} \in \R^{J \times J}$ are orthogonal and $\mvemph{\Sigma} = \diag{\sigma_1,\dots,\sigma_{\min\set{\cardN{I},\cardN{J}}}} \in \R^{I \times J}$ contains the corresponding singular values $\sigma_1 \geq \dots \geq \sigma_{\min\set{\cardN{I},\cardN{J}}} \geq 0$. Now, for the approximant we use $\restrictN{\mvemph{B}}{I \times J} := \mvemph{U}_r \mvemph{\Sigma}_r \mvemph{V}_r^T \in \R^{I \times J}$, where $\mvemph{U}_r \in \R^{I \times r}$, $\mvemph{\Sigma}_r \in \R^{r \times r}$ and $\mvemph{V}_r \in \R^{J \times r}$ are the first $r$ columns of $\mvemph{U}$, $\mvemph{\Sigma}$ and $\mvemph{V}$, respectively. 
Recall from the theory of singular values decompositions (e.g., \cite{Hackbusch_Hierarchical_matrices}) that
\begin{equation*}
\norm[2]{\restrictN{\mvemph{A}^{-1}}{I \times J} - \restrictN{\mvemph{B}}{I \times J}} = \min_{\substack{\mvemph{C} \in \R^{I \times J}: \\ \rank{\mvemph{C}} \leq r}} \norm[2]{\restrictN{\mvemph{A}^{-1}}{I \times J} - \mvemph{C}} = \sigma_{r+1}.
\end{equation*}

In particular, we end up with the following \emph{computable} error bound (cf. \cite[Lemma 6.5.8]{Hackbusch_Hierarchical_matrices})
\begin{equation*}
\norm[2]{\mvemph{A}^{-1} - \mvemph{B}} \cleq \depth{\Tree{N \times N}} \* \max_{(I,J) \in \BPart} \norm[2]{\restrictN{\mvemph{A}^{-1}}{I \times J} - \restrictN{\mvemph{B}}{I \times J}} = \depth{\Tree{N \times N}} \* \max_{(I,J) \in \BPart} \sigma_{r+1}(\restrictN{\mvemph{A}^{-1}}{I \times J}).
\end{equation*}

The numerical example is implemented in MATLAB. For the inversion of the full matrix $\mvemph{A} \in \R^{N \times N}$ we use MATLAB's built-in procedure \texttt{inv(\dots)}. For the singular values decompositions we use \texttt{svds(\dots)}. Recall that an exact matrix inversion needs $\Landau{N^2}$ memory and $\Landau{N^3}$ time to compute, which effectively restricts the maximal feasible problem size to $N \approx 70.000$ on our machine.

\begin{figure}[H]
\begin{center}
\includegraphics[height=6cm, trim=0cm 0cm 0cm 0cm]{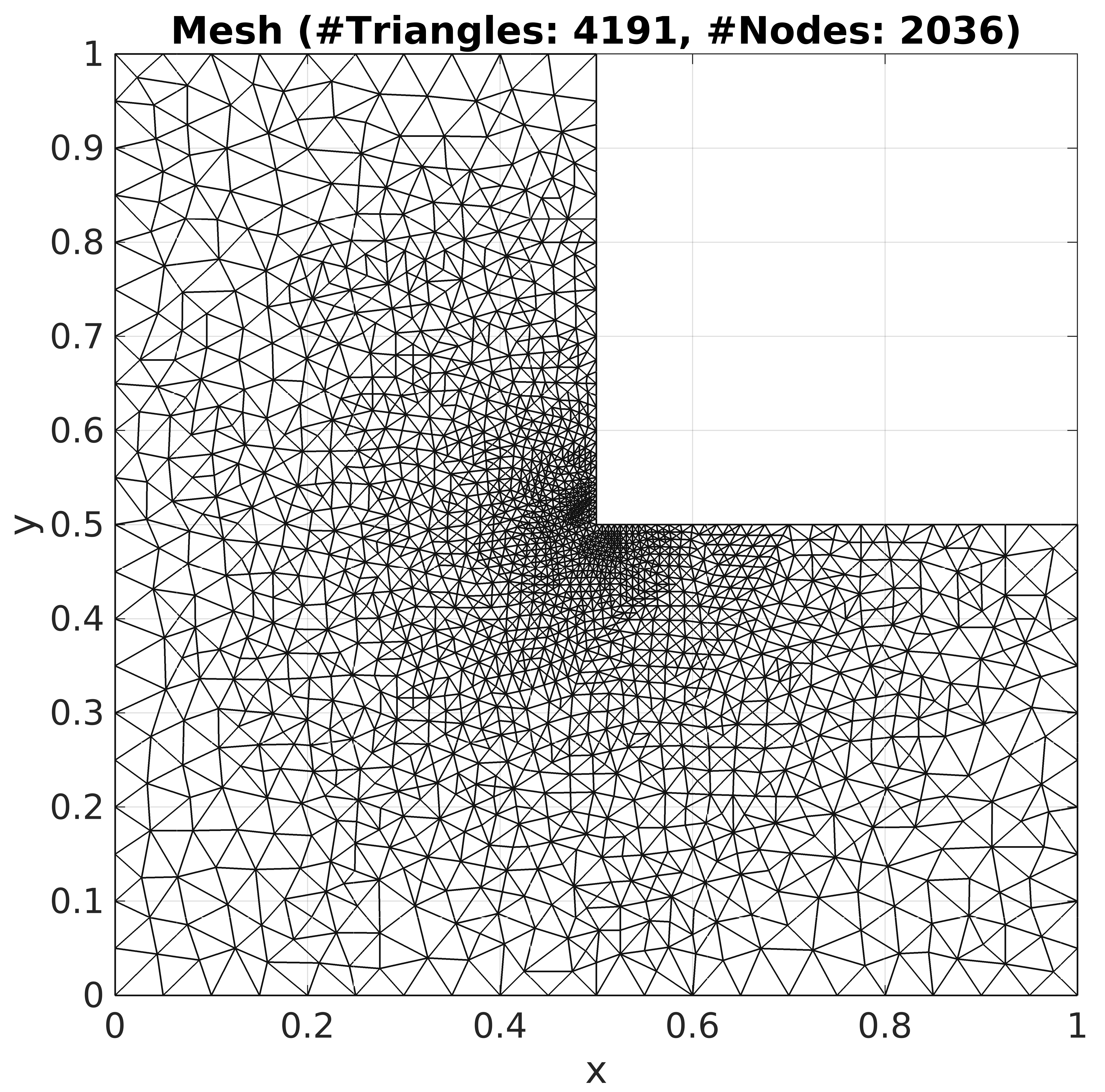}
\includegraphics[height=6cm, trim=0cm 0cm 0cm 0cm]{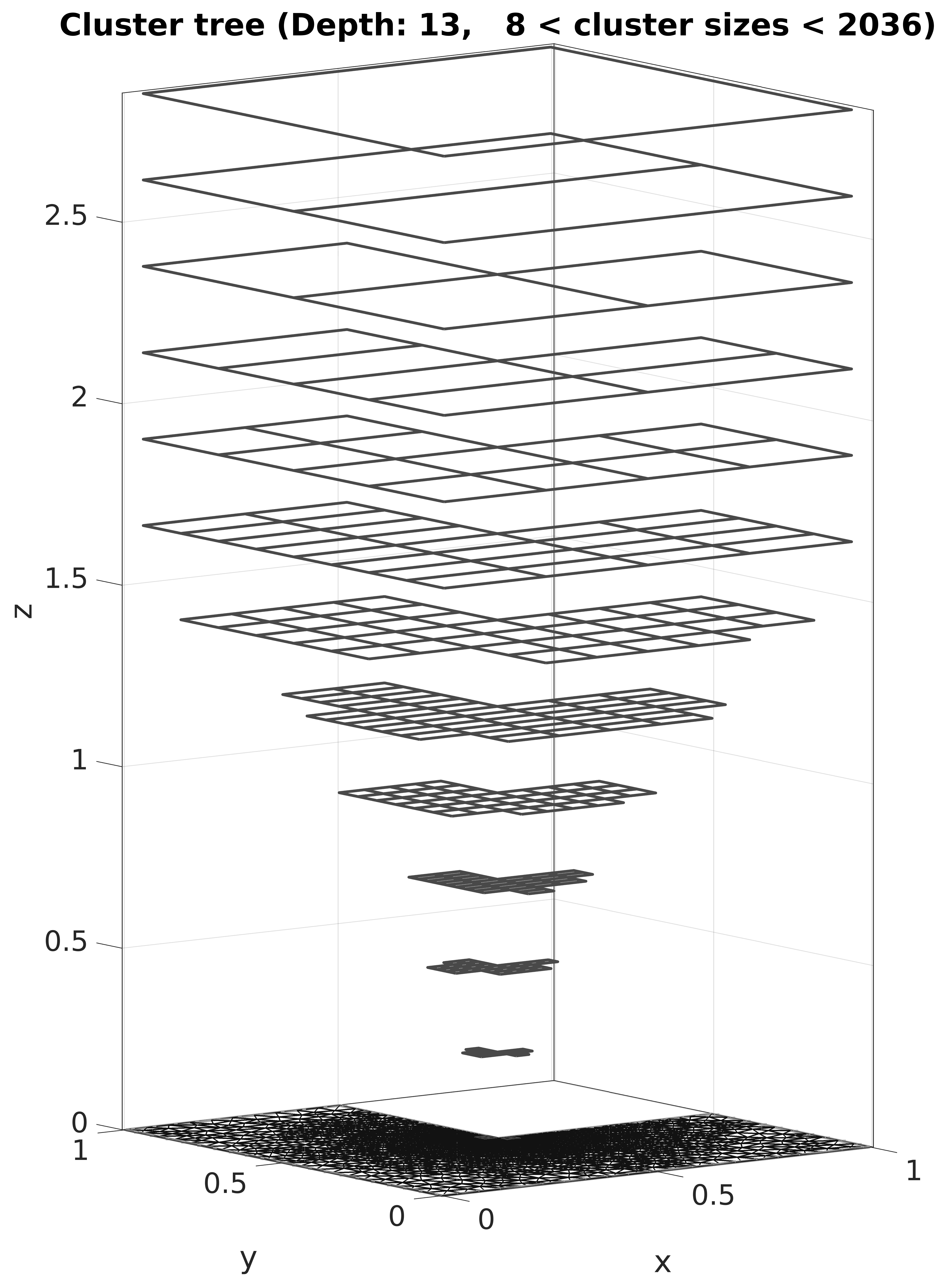}
\includegraphics[height=6cm, trim=0cm 0cm 0cm 0cm]{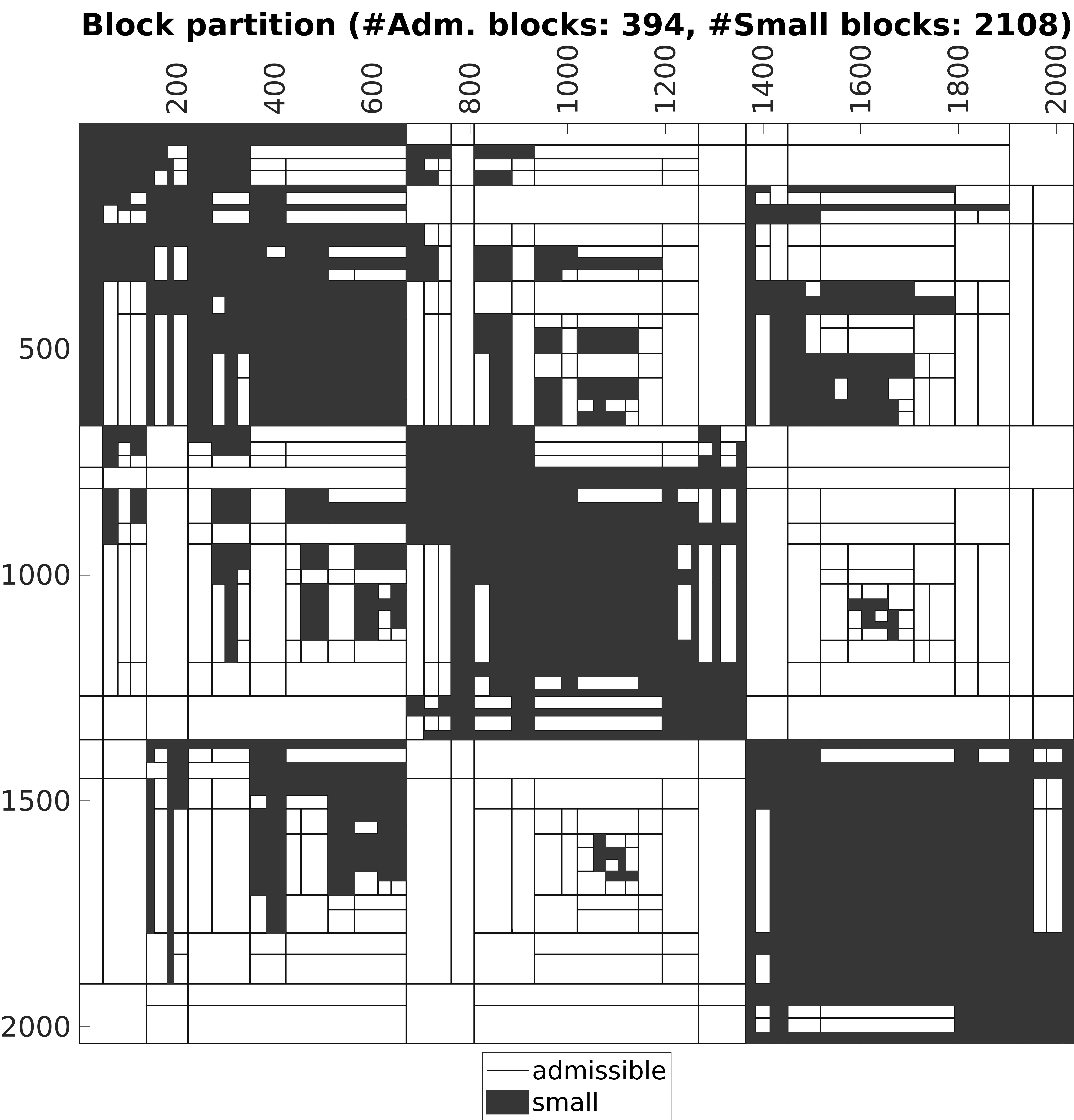}
\caption{The mesh $\Elements$, the cluster tree $\Tree{N}$ and the block partition $\BPart$ for $N \approx 2.000$ degrees of freedom.}
\label{Figure_1}
\end{center}
\end{figure}

In \cref{Figure_1}, we chose $N \approx 2.000$ degrees of freedom. The elements are graded towards the reentrant corner with a grading exponent $\alpha = 5$. The cluster tree $\Tree{N}$ is clearly deeper near the grading center. The block partition $\BPart$ uses sorted indices internally. Only a few admissible blocks are far away from the diagonal, lots of small blocks agglomerate along the diagonal. The sparsity pattern becomes more pronounced as $N \rightarrow \infty$.

\begin{figure}[H]
\begin{center}
\includegraphics[height=6cm, trim=0cm 0cm 0cm 0cm]{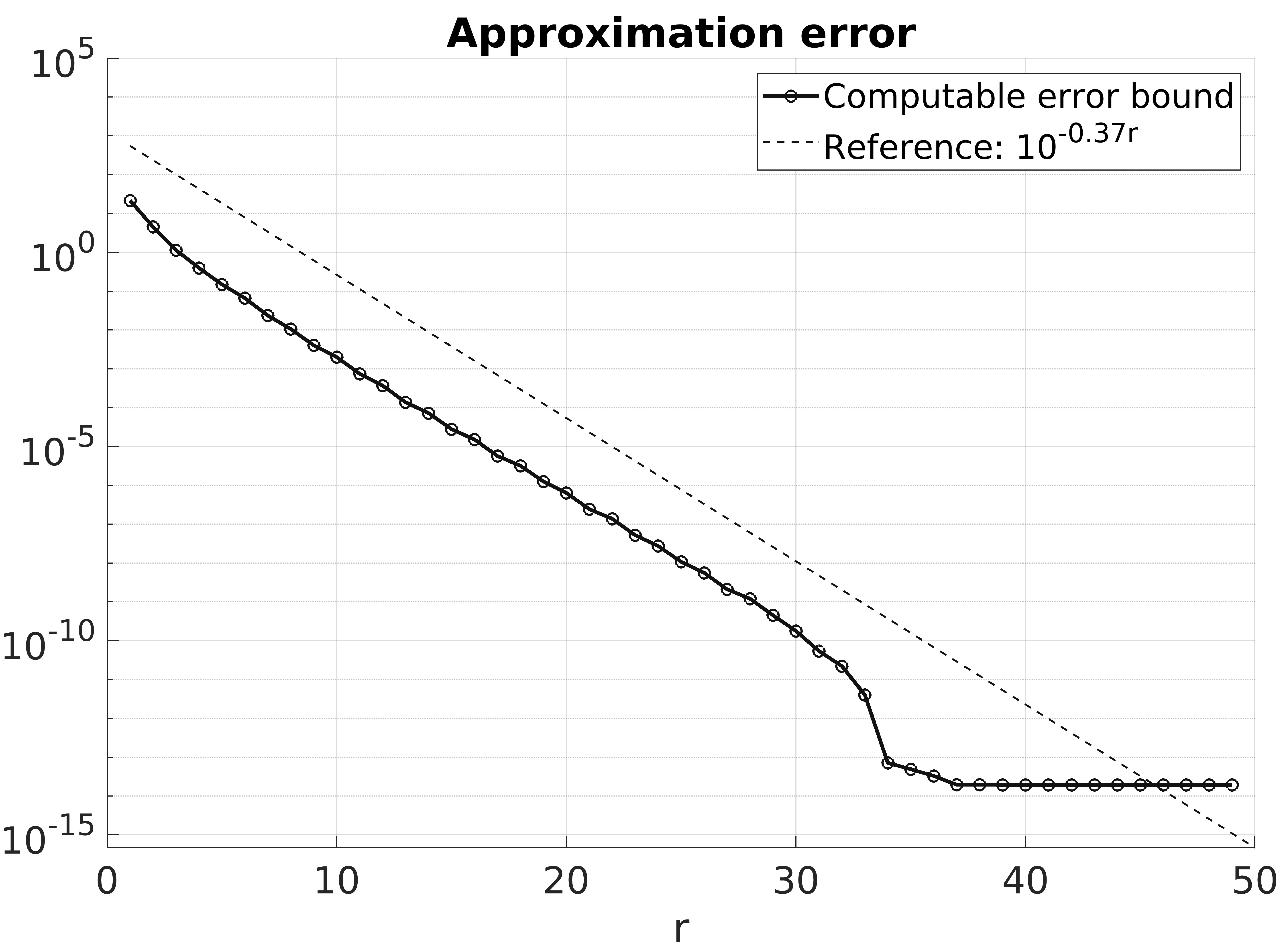}
\includegraphics[height=6cm, trim=0cm 0cm 0cm 0cm]{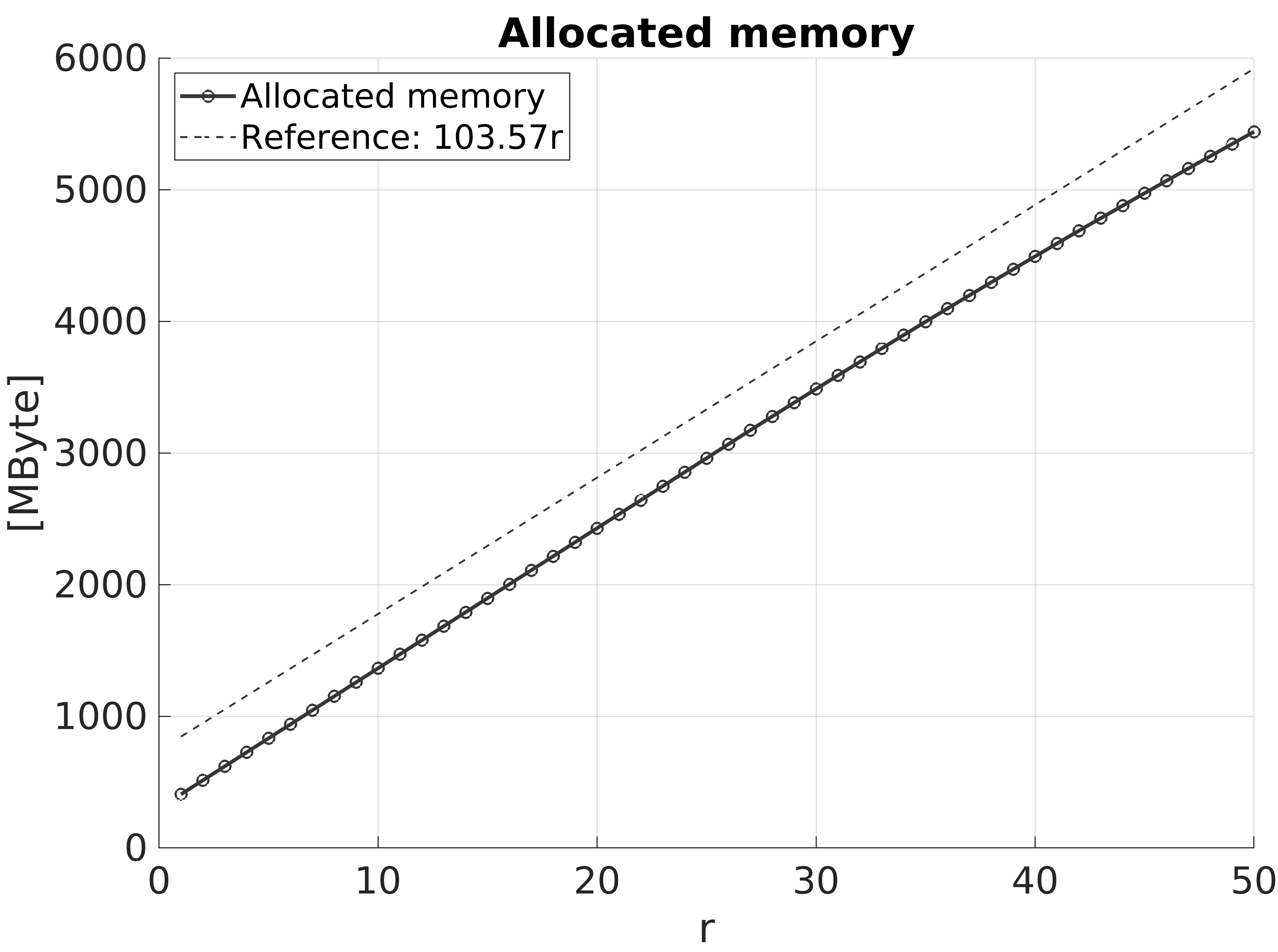}
\caption{Approximation error and memory allocation for $N \approx 72.000$ degrees of freedom.}
\label{Figure_2}
\end{center}
\end{figure}

In \cref{Figure_2}, we chose $N \approx 72.000$ degrees of freedom. The computable error bound from above (for $r \in \set{1,\dots,50}$) is depicted on a linear abscissa and a logarithmic ordinate. The values are below a straight line with slope $-0.37$ indicating an \emph{exponential decay} $\mathrm{error}(r) \cleq 10^{-0.37 r}$. This is even better than the asserted bound from \cref{System_matrix_HMatrix_approx}. The allocated memory in MBytes is plotted on a linear abscissa and a linear ordinate. The values are below a straight line with slope $103.57$ indicating a \emph{polynomial growth} $\mathrm{memory}(r) \cleq r$. Choosing a rank bound $r = 37$, for example, gives an approximation error $\approx 10^{-14}$ and uses $\approx 4.2$ GByte memory. The full system matrix takes $\approx 41.4$ GByte memory.

\bibliographystyle{amsalpha}
\bibliography{References}

\newcommand{\etalchar}[1]{$^{#1}$}
\providecommand{\bysame}{\leavevmode\hbox to3em{\hrulefill}\thinspace}
\providecommand{\MR}{\relax\ifhmode\unskip\space\fi MR }
\providecommand{\MRhref}[2]{%
  \href{http://www.ams.org/mathscinet-getitem?mr=#1}{#2}
}
\providecommand{\href}[2]{#2}
\begin{thebibliography}{DFG{\etalchar{+}}01}

\bibitem[Beb05]{bebendorf05}
M.~Bebendorf, \emph{Efficient inversion of {G}alerkin matrices of general
  second-order elliptic differential operators with nonsmooth coefficients},
  Math. Comp. \textbf{74} (2005), 1179--1199.

\bibitem[Beb07]{bebendorf07}
\bysame, \emph{Why finite element discretizations can be factored by triangular
  hierarchical matrices}, SIAM J. Numer. Anal. \textbf{45} (2007), no.~4,
  1472--1494.

\bibitem[BH03]{BH03}
M.~Bebendorf and W.~Hackbusch, \emph{Existence of {$\mathcal{H}$}-matrix
  approximants to the inverse {FE}-matrix of elliptic operators with
  {$L^{\infty}$}-coefficients}, Numer. Math. \textbf{95} (2003), no.~1, 1--28.

\bibitem[BKP79]{babuska-kellogg-pitkaeranta79a}
I.~Babu{\v s}ka, R.B. Kellogg, and J.~Pitk\"{a}ranta, \emph{Direct and inverse
  error estimates for finite elements with mesh refinements}, Numer. Math.
  \textbf{33} (1979), 447--471.

\bibitem[B{\"o}r10]{boerm10}
S.~B{\"o}rm, \emph{Approximation of solution operators of elliptic partial
  differential equations by {$\mathcal{H}$}- and {$\mathcal{H}^2$}-matrices},
  Numer. Math. \textbf{115} (2010), no.~2, 165--193.

\bibitem[Cia78]{Ciarlet_FEM_Meshes}
P.G. Ciarlet, \emph{The finite element method for elliptic problems},
  North-Holland Publishing Co., Amsterdam-New York-Oxford, 1978, Studies in
  Mathematics and its Applications, Vol. 4.

\bibitem[Cl{\'e}75]{clement75}
Ph. Cl{\'e}ment, \emph{Approximation by finite element functions using local
  regularization}, Rev. Fran\c{c}aise Automat. Informat. Recherche
  Op\'{e}rationnelle S\'{e}r. \textbf{9} (1975), no.~{\rm R}-2, 77--84.

\bibitem[DFG{\etalchar{+}}01]{DFGHS01}
W.~Dahmen, B.~Faermann, I.~G. Graham, W.~Hackbusch, and S.~A. Sauter,
  \emph{Inverse inequalities on non-quasiuniform meshes and application to the
  mortar element method}, Math. Comp. \textbf{73} (2001), 1107--1138.

\bibitem[FMP15]{Faustmann_H_matrices_FEM}
M.~Faustmann, J.M. Melenk, and D.~Praetorius, \emph{H-matrix approximability of
  the inverses of {FEM} matrices}, Numer. Math. \textbf{131} (2015), no.~4,
  615--642.

\bibitem[FMP16]{FMP16}
M.~Faustmann, J.~M. Melenk, and D.~Praetorius, \emph{Existence of {${\mathcal
  H}$}-matrix approximants to the inverse of {BEM} matrices: the simple-layer
  operator}, Math. Comp. \textbf{85} (2016), 119--152.

\bibitem[FMP17]{FMP17}
\bysame, \emph{Existence of {${\mathcal H}$}-matrix approximants to the inverse
  of {BEM} matrices: the hyper-singular integral operator}, IMA J. Numer. Anal.
  \textbf{37} (2017), no.~3, 1211--1244.

\bibitem[GH03]{GH03}
L.~Grasedyck and W.~Hackbusch, \emph{Construction and arithmetics of
  {$\mathcal{H}$}-matrices}, Computing \textbf{70} (2003), no.~4, 295--334.

\bibitem[GHK08]{GHK08}
L.~Grasedyck, W.~Hackbusch, and R.~Kriemann, \emph{Performance of
  {$\mathcal{H}$}-{LU} preconditioning for sparse matrices}, Comput. Methods
  Appl. Math. \textbf{8} (2008), no.~4, 336--349.

\bibitem[GHLB04]{Grasedyck_Clustering}
L.~Grasedyck, W.~Hackbusch, and S.~Le~Borne, \emph{Adaptive geometrically
  balanced clustering of h-matrices}, Computing \textbf{73} (2004), no.~1,
  1--23.

\bibitem[GKLB08]{GKL08}
L.~Grasedyck, R.~Kriemann, and S.~Le~Borne, \emph{Parallel black box {$\mathcal
  H$}-{LU} preconditioning for elliptic boundary value problems}, Comput. Vis.
  Sci. \textbf{11} (2008), no.~4-6, 273--291.

\bibitem[GR97]{greengard-rokhlin97}
L.~Greengard and V.~Rokhlin, \emph{A new version of the fast multipole method
  for the {L}aplace in three dimensions}, Acta Numerica 1997, Cambridge
  University Press, 1997, pp.~229--269.

\bibitem[Gra01]{grasedyck01}
L.~Grasedyck, \emph{Theorie und {A}nwendungen {H}ierarchischer {M}atrizen},
  Ph.D. thesis, Universit{\"a}t Kiel, 2001.

\bibitem[Hac99]{hackbusch99}
W.~Hackbusch, \emph{A sparse matrix arithmetic based on
  {$\mathcal{H}$}-matrices. {I}ntroduction to {$\mathcal{H}$}-matrices},
  Computing \textbf{62} (1999), no.~2, 89--108.

\bibitem[Hac15]{Hackbusch_Hierarchical_matrices}
\bysame, \emph{Hierarchical matrices: algorithms and analysis}, Springer Series
  in Computational Mathematics, vol.~49, Springer, Heidelberg, 2015.

\bibitem[Rok85]{rokhlin85}
V.~Rokhlin, \emph{Rapid solution of integral equations of classical potential
  theory}, J. Comput. Phys. \textbf{60} (1985), 187--207.

\bibitem[Sch98]{schneider98}
R.~Schneider, \emph{Multiskalen- und {W}avelet-{M}atrixkompression:
  {A}nalysisbasierte {M}ethoden zur effizienten {L}\"osung gro\ss{}er
  vollbesetzter {G}leichungssysteme}, Advances in Numerical Mathematics,
  Teubner, 1998.

\bibitem[TW03]{tausch-white03}
J.~Tausch and J.~White, \emph{Multiscale bases for the sparse representation of
  boundary integral operators on complex geometry}, SIAM J. Sci. Comput.
  \textbf{24} (2003), no.~5, 1610--1629.

\bibitem[vPSS97]{PSS97}
T.~von Petersdorff, Ch. Schwab, and R.~Schneider, \emph{Multiwavelets for
  second-kind integral equations}, SIAM J. Numer. Anal. \textbf{34} (1997),
  no.~6, 2212--2227.

\end{thebibliography}

\end{document}